\documentclass[11pt]{article}

\usepackage[a4paper,margin=1in]{geometry}
\usepackage{amsmath, amssymb, amsfonts, amsthm, mathtools}
\usepackage{enumitem}
\usepackage{graphicx}
\usepackage{subcaption}
\usepackage{tikz}
 \usetikzlibrary{calc}
 \usetikzlibrary{positioning}
\usepackage{booktabs}
\usepackage{hyperref}
\usepackage{xcolor}
\usepackage{tabularx}
\usepackage{pifont}
\usepackage{dsfont}

\newcommand{\cmark}{\ding{51}}
\newcommand{\xmark}{\ding{55}}

\usepackage{titlesec}

\titleformat{\subsection}[runin]
  {\bfseries}
  {\thesubsection}
  {0.5em}
  {}

\hypersetup{
    colorlinks=true,
    linkcolor=blue,
    citecolor=blue,
    urlcolor=blue
}

\theoremstyle{plain}
\newtheorem{theorem}{Theorem}[section]
\newtheorem{lemma}[theorem]{Lemma}
\newtheorem{proposition}[theorem]{Proposition}

\newtheorem{conjecture}[theorem]{Conjecture}

\theoremstyle{definition}

\newtheorem{remark}[theorem]{Remark}

\numberwithin{equation}{section}

\title{Models for species evolution with random deaths}

\author{
    Peter Braunsteins\thanks{The University of New South Wales, Sydney NSW 2052, Australia. \texttt{ p.braunsteins@unsw.edu.au}}
    \and
    Joseph Rolfe\thanks{The University of New South Wales, Sydney NSW 2052, Australia. \texttt{joseph.rolfe@student.unsw.edu.au}}
}

\date{\today}

\begin{document}

\maketitle


\begin{abstract}

We consider three discrete-time models for species evolution. In all three models, at each time step $n$, with probability $p$, a species is born with an independent $\text{Uniform}[0,1]$ fitness value and, with probability $1-p$, a species is killed. The mechanism for selecting which species to kill when a death occurs distinguishes the three models: in the first model, the least fit species is always killed; in the second model, with probability $r$, the least fit species is killed and, with probability $1-r$, a species chosen uniformly from the population is killed; in the third model, with probability $r$, the least fit species is killed and, with probability $1-r$, the species with the largest fitness less than an independent $\text{Uniform}[0,1]$ outcome is killed. We establish asymptotic results as $n \to \infty$ for the three models. These results demonstrate that small changes to the death mechanism of the model can lead to vastly different asymptotic behaviour. To prove our results we develop a novel approach that relies on coupling arguments and mean-field limits. 
\end{abstract}

\section{Introduction}

\label{ch:introduction}

Raup \cite{Paleontological Extinction} estimates that more than 99.9\% of all animal species that have ever existed are now extinct. He also observes that although some extinctions are selective, many appear to be effectively random, in the sense that survivors exhibit no clear superiority over those that disappear. In this paper, we compare three models for species evolution with different random extinction mechanisms: one that always eliminates the least fit species, and two others in which a species other than the least fit may be eliminated. 
By comparing the asymptotic behaviour of the three models, our goal is
to show that small changes in the death mechanism of the model can lead to vastly different behaviour. 

We first consider the {\em Guiol–Machado–Schinasi (GMS)} process introduced in \cite{Simple Process}. This process starts with no species alive. At each time step $n \in \mathbb{N}$, with probability $p$, a new species is born and assigned an independent $\mathrm{Uniform}[0,1]$ fitness value; otherwise, with probability $q:=1-p$, the least fit species dies. The key results in the context of this paper, established in \cite{Simple Process,Bessel Link}, are the following. Suppose $p > q$:
\begin{enumerate}[label=(\arabic*),itemsep=1em]
\item[(1-a)] If $f < f^*_c := \frac{q}{p}$ then, with probability 1, there are infinitely many times $n$ such that the number of species with fitness less than $f$ is 0, whereas if $f > f^*_c$ then, with probability 1, the number of species with fitness less than $f$ increases linearly in $n$ as $n \to \infty$.
\item[(1-b)] The distribution of the fitness values of the species alive at time $n$ converges to the uniform distribution on $(f^*_c,1)$ as $n \to \infty$.
\item[(1-c)] Species with fitness $f > f^*_c$ have a strictly positive probability of surviving indefinitely.
\end{enumerate}
These results, which are restated rigorously in Section \ref{sec:GMS}, highlight the importance of the critical threshold $f^*_c=\frac{q}{p}$. 
For $p>q$ and at a large time $n$, the number of species alive is approximately $(p-q)n$ and the total number of species ever born is approximately $pn$; thus approximately $\left( \frac{p-q}{p} \right)\times 100 \%$ of all species born up to a large time $n$ are still alive at time $n$. This means Raup's estimate that over $99.9\%$ of all species the have ever existed are now extinct aligns with $p \downarrow \frac{1}{2}$, in which case, $f^*_c \uparrow 1$. Hence, from (1-b), the GMS model suggests that the species alive when $n$ is large (i.e., species alive today) have fitness very close to 1 (i.e., amongst the fittest that have ever existed). This is illustrated in the left panels of Figure \ref{fig:intro_gms_process_behaviour}, where $p=0.50025$, so that approximately $99.9\%$ of species that have existed are extinct at time $n=10^7$, and $f^*_c \approx 0.999$, which implies that the species alive at a large time $n$ are overwhelmingly in the top $0.1\%$ of fitness values.

\begin{figure}[h!]
\centering
\includegraphics[width=0.9\linewidth]{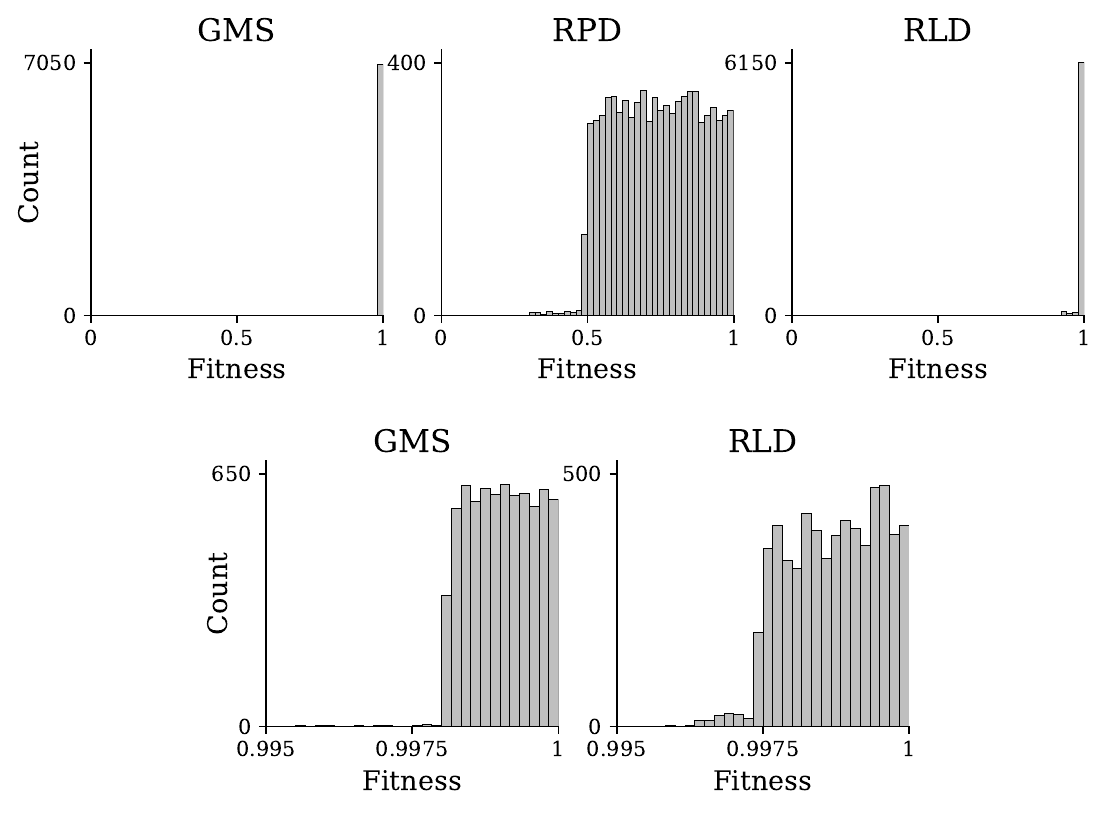}
\caption{Fitness histograms for the GMS (left), RPD, and RLD (right) models when $n=10^7$, $p=0.50025$, and $r=0.5$, for fitness values in $[0,1]$ (top row) and $[0.99,1]$ (bottom row).}
\label{fig:intro_gms_process_behaviour}
\end{figure}

The second model we consider reflects the inherently random nature of extinction. It was first introduced in \cite{Influenza Process} and we refer to it as the \emph{random population deaths (RPD)} process. It extends the death mechanism in the GMS process as follows:
\begin{itemize}
\item At each death event, with probability $r$, the least fit species dies; otherwise, with probability $1-r$, a species is chosen uniformly at random from the population to die.
\end{itemize}
The only established result for the RPD process when $p>q$ concerns the maximal fitness among all species alive at time $n$, which converges to 1 as $n \to \infty$ \cite[Theorem 4]{Influenza Process}. It is also considered in \cite[Remark 2.5]{Actual Process}, where it is conjectured that there is critical threshold for the RPD process that plays a similar role to $f^*_c$ in the GMS process, and that this threshold is $f^\dagger_c:=\tfrac{qr}{p}$; however, the authors were unable to prove this result and do not mention that the formula does not hold when $p < q$. The lack of results on the RPD process is perhaps because, as we explain in Section \ref{sec:Ln-RPD}, the proofs require techniques that fundamentally differ from those used for the GMS process. We develop a new proof approach that uses ideas from interacting particle systems \cite{L85} such as coupling arguments and mean-field limits.
Suppose $p>q$, we show that:
\begin{enumerate}[label=(\arabic*),itemsep=1em]
\item[(2-a)] If $f<f^\dagger_c:=\tfrac{qr}{p}$ then, with probability 1, there are infinitely many times $n$ such that the number of species with fitness less than $f$ is 0, whereas if $f > f^\dagger_c$ then, with probability 1, the number of species with fitness less than $f$ increases linearly in $n$ as $n \to \infty$.
\item[(2-b)] The distribution of the fitness values of the species alive at time $n$ converges to the uniform distribution on $(f^\dagger_c,1)$ as $n \to \infty$.
\item[(2-c)] Species with \textit{any} fitness $f$ eventually die with probability 1.
\end{enumerate}
These results are restated rigorously in Section \ref{sec:RPD-KR}. In contrast to the GMS model, observe that when $p \downarrow \frac{1}{2}$, we have $f^\dagger_c \uparrow r$. In the centre top panel of Figure~\ref{fig:intro_gms_process_behaviour}, $p=0.50025$ so we again have that approximately $99.9\%$ of species that have existed are extinct at time $n=10^7$, and we take $r=0.5$. In this case $f^\dagger_c = 0.4995$, so that in the long term the fitness distribution of the species alive for large $n$ is uniform on the top $50.05\%$ of fitness values. Consequently, unlike in the GMS model, the species alive at a large time $n$ \emph{are} \emph{not} amongst the fittest that have ever existed.

The third model we consider illustrates that contrasting behaviour may arise when the random death mechanism is modified. We refer to this process as the \emph{random location deaths} (RLD) process, which has not previously appeared in the literature. It extends the death event of the GMS model as follows:
\begin{itemize}
    \item At each death event, with probability $r$, the least-fit species dies; otherwise, with probability $1-r$, independently sample $U_n \sim \mathrm{Uniform}[0,1]$ and delete the species whose fitness is the largest value strictly less than $U_n$ (if no such species exists, the least-fit species dies).
\end{itemize}
As we explain in Section \ref{sec:RLD-Ln}, the RLD process is more analytically challenging than the RPD process. Consequently, in (3-a) and (3-b) below, the results in \textit{italic} are conjectured; the reasoning behind these conjectures is explained in Section \ref{sec:RLD-Ln}.
Suppose $p>q$:
\begin{enumerate}[label=(\arabic*),itemsep=1em]
    \item[(3-a)] If $f<f^\ddagger_c:=\frac{qr}{p-q(1-r)}$ then, with probability 1, there are infinitely many times $n$ such that the number of species with fitness less than $f$ is 0, \textit{whereas if $f>f^\ddagger_c$ then, with probability 1, the number of species with fitness less than $f$ increases linearly in $n$ as $n \to \infty$.}
    \item[(3-b)] \textit{The distribution of the fitness values of the species alive at time $n$ converges to the uniform distribution on $(f^\ddagger_c,1)$ as $n \to \infty$.}
    \item[(3-c)] Species with any fitness $f<1$ eventually die with probability 1.
\end{enumerate}
These results are restated rigorously in Section \ref{sec:RLD-KR}. Like the GMS process but unlike the RPD process, observe that when $p \downarrow \frac{1}{2}$, we have $f^\ddagger_c \uparrow 1$. In the right panels of Figure~\ref{fig:intro_gms_process_behaviour}, we again take $r=0.5$ and $p=0.50025$. In this case, $f^\ddagger_c = 0.998$, so that the long term fitness distribution of the species alive for large $n$ is uniform on the top $0.2\%$ of fitness values. That is, like the GMS process but unlike the RPD process, the species alive at time $n$ \emph{are} amongst the fittest that have ever existed. In Remark \ref{rem:no_conj} we explain that this observation is supported theoretically by the first part of (3-a), i.e., it does not rely on any conjectures. We provide a more detailed comparison between the GMS, RPD, and RLD processes in Section \ref{sec:discussion}.

Our work belongs to the literature on extensions of the GMS model. The GMS model was introduced in \cite{Simple Process} because it displays similar properties to the Bak-Sneppen model \cite{BakSneppen1993} which is more challenging to analyse mathematically (c.f. \cite{MZ03}). It has been extended to allow for continuous time \cite{Bessel Link}, multiple births/deaths simultaneously \cite{B13,MS12} and discrete fitness distributions \cite{Actual Process}. Results similar to (1-a)--(1-c) have been established for these models, and follow primarily through extensions of the arguments for the GMS process outlined in Section \ref{sec:GMS_Ln}. Different questions, such as the distribution of the maximal fitness \cite{StrongestFitness,Phylogenetic Trees}, and finer results such as laws of iterated logarithms \cite{BMR11} have also been addressed, but we do not consider these here. Another variant of the GMS process, which introduces a mechanism whereby, with some probability, a new species inherits the fitness of an existing one, is considered in \cite{Quasispecies Model, Subspecies Model}; however, we note that in these models questions concerning $f_c$ reduce to those of the GMS model with a modified birth probability. In addition, \cite{EventBasedExtinction} introduces a \emph{mass-death} mechanism, in which all species with fitness below a random threshold (sampled from a fixed distribution) are simultaneously removed; however, this modification leads to fundamentally different behaviour in the sense that (1-a)--(1-c) do not naturally generalise to this setting. 

The paper is organised as follows. The GMS, RPD, and RLD processes are considered in Sections \ref{sec:GMS}, \ref{sec:RPD}, and \ref{section:rld_process} respectively. Section \ref{sec:discussion} compares the three models and proposes future directions for research. Finally, Section \ref{sec:proofs} contains the proofs of our results.

\paragraph{Notation:} The following notation is used to describe the three models. Let $\mathcal{M}_n$ be the set of fitness values of species alive at time $n$, and let $M_n$ denote the number of species alive at time $n$, i.e., $M_n=|\mathcal{M}_n|$, where $|\cdot|$ denotes the cardinality. For a fixed $f \in [0,1]$, let $\mathcal{L}_n(f)$ be the set of fitness values of all species alive at time $n$ with fitness less than $f$, and let $L_n(f)\equiv |\mathcal{L}_n(f)|$. For example, if at time $n=5$ we have $\mathcal{M}_5=\{0.1, 0.6, 0.8\}$, then $M_5=3$, $\mathcal{L}_5(0.7)=\{0.1,0.6\}$, and $L_5(0.7)=2$. Let 
\[
\mathcal{B}(f) = \{n \geq 0: L_n(f) = 0\}
\]
be the set of times at which $L_n(f)$ is 0, and denote the elements of $\mathcal{B}(f)$ by $0=B_0<B_1 < B_2 <\dots$. Finally, we note that in all three models $\{M_n : n \in \mathbb{N}_0\}$ is a random walk with a reflecting barrier at 0 that increases by 1 with probability $p$ and decreases by 1 with probability $q$.

\section{The GMS Model}\label{sec:GMS}

\subsection{Model.} The GMS process is characterised by the following dynamics: at time $n = 0$ no species exist, and at each time $n \geq 1$: 
\begin{itemize}
    \item With probability $p$, a new species is born and assigned an independent fitness value $f_n \sim \text{Uniform}[0,1]$. Otherwise,
    \item with probability $q =1 - p$, the species with the lowest fitness is eliminated (if no species are alive, nothing happens). 
\end{itemize} 

 \subsection{Key results.}\label{sec:GMS-KR}
The properties (1-a)--(1-c) of the GMS model that were described in Section \ref{ch:introduction} are formalised in the following theorems. They are written in a manner that helps us to compare them with the corresponding results in Sections \ref{sec:RPD} and \ref{section:rld_process}. 
In preparation, recall the notation introduced at the end of Section \ref{ch:introduction}.

\begin{theorem}[{\cite[Theorem 1a]{Simple Process}}]
Let $f^*_c = \frac{q}{p}$ and suppose $p>q$.
    \begin{itemize}
        \item[(i)] If $f < f^*_c$, then $|\mathcal{B}(f)|= \infty$ a.s. and $\mathbb{E}[B_{i+1} - B_i] = \frac{f^*_c}{f^*_c-f}<\infty$ for all $i$.
        \item[(ii)] If $f = f^*_c$, then $|\mathcal{B}(f)|= \infty$ a.s. and $\mathbb{E}[B_{i+1} - B_i]=\infty$ for all $i$. 
        \item[(iii)] If $f > f^*_c$, then $|\mathcal{B}(f)| < \infty$ a.s..
    \end{itemize}
    \label{theorem:gms_theorem}
\end{theorem}

\begin{theorem}[{\cite[Theorem 1b]{Simple Process}}]
Let $f^*_c = \frac{q}{p}$ and suppose $p>q$. Then for $f^*_c \leq a < b \leq 1$,
        \[
         \frac{L_n(b) - L_n(a)}{n} \xrightarrow{a.s.} (p-q) \left(\frac{b-a}{1-f^*_c} \right), \qquad \text{as } n \to \infty.
        \]
        \label{theorem:gms_rn_distribution}
\end{theorem}

\begin{theorem}[{\cite[Theorem 2.1]{Bessel Link}}]
Let $f^*_c = \frac{q}{p}$, $f\in [0,1]$, and suppose $p>q$. If the number of species currently alive with fitness below $f$ is $k$, then the probability that a species with fitness $f$ survives forever is
\[
\begin{cases}
0 \qquad &\text{if } f < f^*_c, \\
1 - \left(\tfrac{f^*_c}{f}\right)^{k+1} \qquad &\text{if } f \geq f^*_c.
\end{cases}
\]
    \label{theorem:gms_survival}
\end{theorem}
We note that while Theorem \ref{theorem:gms_survival} is not stated directly in \cite[Theorem 2.1]{Bessel Link}, it is an immediate consequence. We further note that Theorem \ref{theorem:gms_rn_distribution} implies that the empirical distribution of species fitnesses in $\mathcal{M}_n$ converges to the uniform distribution on $[f^*_c,1]$ as $n \to \infty$.

\subsection{Representation of $L_n(f)$.}\label{sec:GMS_Ln} 
The property of the GMS model used to prove these key results is that $L_n(f)$ can be represented as the simple one-dimensional Markov birth–death chain. This chain is illustrated in Figure \ref{fig:|L_n(f)|_gms} and can be identified as a random walk with a reflecting barrier at 0. To obtain these transition probabilities we note that, when $L_n(f)>0$: 
\begin{itemize}
\item[(1)] \( L_n(f)\) increases by 1 if a species is born and its fitness is less than or equal to \( f \), which occurs with probability \( pf \); 
\item[(2)] \( L_n(f)\) decreases by 1 if a species with fitness less than or equal to \( f \) dies, and because death events always target the least-fit species and \( L_n(f)\) always contains the least fit whenever \( L_n(f)> 0 \), this occurs with probability \( q \); 
\item[(3)] \( L_n(f)\) remains the same when a new species is born with fitness greater than \( f \), which happens with probability \( p(1 - f) \). 
\end{itemize}
When $L_n(f)=0$, death events do not affect $L_n(f)$, so \( L_n(f)\) has an additional probability $q$ of remaining the same.
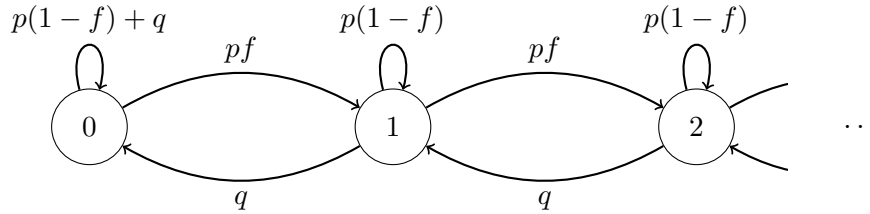
\begin{figure}[h!]
    \centering
    \begin{tikzpicture}[
        node distance=3cm,
        state/.style={circle, draw, minimum size=1cm}
    ]

        \node[state] (A0) at (0,0) {0};
        \node[state, right=of A0] (A1) {1};
        \node[state, right=of A1] (A2) {2};

        \begin{scope}
            \clip (-1.5,-2) rectangle ($(A2)+(1.2,2)$);

            \draw[->, thick] (A0) to[out=30, in=150] node[above] {$pf$} (A1);
            \draw[->, thick] (A1) to[out=30, in=150] node[above] {$pf$} (A2);
            \draw[->, thick] (A2) to[out=30, in=150] 
                node[midway, above] {$pf$} ($(A2)+(3,0.25)$);

            \draw[->, thick] (A1) to[out=210, in=330] node[below] {$q$} (A0);
            \draw[->, thick] (A2) to[out=210, in=330] node[below] {$q$} (A1);
            \draw[->, thick] ($(A2)+(3,-0.25)$) 
                to[out=210, in=330] node[midway, below] {$q$} (A2);

            \draw[->, thick] (A0) edge[loop above, looseness=8] node {$p(1-f)+q$} (A0);
            \draw[->, thick] (A1) edge[loop above, looseness=8] node {$p(1-f)$} (A1);
            \draw[->, thick] (A2) edge[loop above, looseness=8] node {$p(1-f)$} (A2);
        \end{scope}

        \node at ($(A2)+(2.2,0)$) {\dots};
    \end{tikzpicture}
    \caption{Transition probabilities of $L_n(f)$ in the GMS process.}
    \label{fig:|L_n(f)|_gms}
\end{figure}

The following sketches of proof illustrate the connection between this Markov chain and the key results above, and are included to highlight the additional challenge of analysing the RPD and RLD models in Sections \ref{sec:RPD} and \ref{section:rld_process}.
\begin{itemize}
\item[\emph{(a)}] Theorem \ref{theorem:gms_theorem} follows by recurrence/transience properties of the Markov chain $L_n(f)$; specifically, when it is recurrent \emph{(i)}, null-recurrent \emph{(ii)}, and transient \emph{(iii)}.
\item[\emph{(b)}] Theorem \ref{theorem:gms_rn_distribution} is more involved than the other two, but follows because: (1) since $L_n(f)$ is a random walk with a reflecting barrier at 0 it has the standard representation 
\begin{equation}\label{eqn:decomp}
L_n(f) \stackrel{d}{=} S_n - \inf_{0 \leq k \leq n} S_k,
\end{equation}
where $S_n = \sum_{i=1}^n X_i$ and $\{X_i\}$ are independent random variables with $\mathbb{P}(X_i=-1)=q$, $\mathbb{P}(X_i=0)=p(1-f)$ and $\mathbb{P}(X_i = 1) =pf$; (2) $\frac{S_n}{n} \stackrel{a.s.}{\to} pf-q$ by the SLLN; (3) For $f>f^*_c$, $\mathbb{E}(X_i)>0$ which is used to show $\frac{1}{n}\inf_{0\leq k \leq n} S_k \stackrel{a.s.}{\to} 0$; 
(4) This implies $\frac{L_n(f)}{n} \stackrel{a.s.}{\to} pf-q = (p-q) \frac{f-f^*_c}{1-f^*_c}$;
(5) By applying Slutsky's Theorem we then obtain the result. 
\item[\emph{(c)}] Theorem \ref{theorem:gms_survival} follows by noting that the species with fitness $f$ eventually dies if and only if $L_n(f)$ eventually hits 0. Because $L_n(f)$ is a simple random walk (with reflecting barrier), the probability $L_n(f)$ eventually hits 0 is the well-known gambler's ruin probability, i.e., the probability that a gambler with $k+1$ dollars eventually loses all their money when they bet 1 dollar per turn.
\end{itemize}

\section{Random population deaths}\label{sec:RPD}

\subsection{Model.} The random population deaths (RPD) process is characterised by the following dynamics: at time $n=0$ no species exist and at each time $n \geq 1$
\begin{itemize}
    \item With probability $p$, a new species is born and assigned an independent fitness value $f_n \sim \text{Uniform}[0,1]$. Otherwise, 
    \item with probability $q = 1-p$, one species dies, in which case:
          \begin{itemize} 
            \item with probability $r$, the species with lowest fitness value dies. Otherwise,
            \item with probability $1 - r$, a species is selected uniformly at random from the set of currently living species to die (i.e., with equal probability from the set of all living species).
          \end{itemize}
\end{itemize}

\subsection{Key results.}\label{sec:RPD-KR}

The properties (2-a)--(2-c) of the RPD model that were described in Section \ref{ch:introduction} are formalised in the following theorems. In preparation, recall the notation introduced at the end of Section \ref{ch:introduction}.
\begin{theorem}
\label{theorem:RPD_L_n}
Let $f^\dagger_c = \tfrac{qr}{p}$ and suppose $p>q$.
\begin{itemize}
\item[(i)] If \( f < f^\dagger_c \), then $|\mathcal{B}(f)| = \infty$ a.s. and $\mathbb{E}(B_{i+1}-B_i) \leq  \frac{f^\dagger_c}{f^\dagger_c-f} <\infty$ for all $i$.  
\item[(ii)] If $f=f^\dagger_c$ then $|\mathcal{B}(f)|=\infty$ a.s.  and $\mathbb{E}(B_{i+1}-B_i) \to \infty$ as $i \to \infty$.
\item[(iii)] If \( f > f^\dagger_c \), then $|\mathcal{B}(f)| <\infty$ a.s.
\end{itemize}
\end{theorem}

\begin{theorem}
Let $f^\dagger_c=\frac{qr}{p}$. Then for any $f^\dagger_c\le a<b\le1$,
\[
\frac{L_n(b)-L_n(a)}{n}
\stackrel{\mathbb{P}}{\to}
(p-q)\left(\frac{b-a}{1-f^\dagger_c}\right), \qquad \text{as } n \to \infty.
\]
\label{theorem:RPD_Uniformity}
\end{theorem}

\begin{theorem}
    If $r<1$ then the probability that any species survives forever is $0$.
    \label{theorem:RPD_survival}
\end{theorem}

Note the similarity between the statements of Theorems \ref{theorem:RPD_L_n} and \ref{theorem:RPD_Uniformity} with Theorems \ref{theorem:gms_theorem} and \ref{theorem:gms_rn_distribution}. The key difference is the formula for $f_c$, which leads to the contrasting behaviour in Figure \ref{fig:intro_gms_process_behaviour}.

\begin{remark}\label{rem:null_r}
Under the conditions of Theorem \ref{theorem:RPD_L_n} \emph{(ii)}, it is unclear to the authors whether the stronger assertion $\mathbb{E}(B_{i+1}-B_i) = \infty$ for all $i$ holds.
\end{remark}

\subsection{Representation of $L_n(f)$.}\label{sec:Ln-RPD} Recall that the results on the GMS model in Section~\ref{sec:GMS} rely on the fact that $L_n(f)$ is a simple one–dimensional birth–death chain. When we attempt to construct a similar birth–death chain for the RPD model, we run into problems. 
When $L_n(f)>0$, the transition probabilities for $L_n(f)$ are determined by the following disjoint events:
\begin{itemize}
    \item A birth happens in $[0,f]$ (i.e., $L_n(f)$ increases by 1) with probability $pf$;
    \item A death happens that targets the least-fit (i.e., $L_n(f)$ decreases by 1) with probability $qr$;
    \item A death targets a species in $[0,f]$ uniformly at random from alive species (i.e., $L_n(f)$ decreases by 1) with probability $q(1-r)\frac{L_n(f)}{M_n}$.
\end{itemize}
Due to the presence of $M_n$ in the final point, the transition probabilities of $L_n(f)$ cannot be represented solely as a function of the current state of $L_n(f)$. Thus $L_n(f)$ cannot be represented as a one-dimensional Markov chain. 

To construct a Markov chain, we extend the state space so that it includes $M_n$. We thus consider the stochastic process $\{(M_n,L_n(f)): n \in \mathbb{N}_0\}$ with state space $\mathcal{S} = \{(m,\ell) \in \mathbb{N}_0^2: m \geq \ell\}$. Let
\[
P(m',\ell'):=\mathbb{P}\left.\bigg(M_{n+1}=m', L_{n+1}(f)=\ell'\; \right| \; M_n=m, L_{n}(f)=\ell\bigg).
\]
The transition probabilities of $(M_n,L_n(f))$ that correspond to births in $[0,f]$ and $(f,1]$ are
\begin{equation}\label{eqn:NL1}
P(m+1,\ell+1) = pf \quad \text{and} \quad P(m+1, \ell) = p(1-f),
\end{equation}
those that correspond to deaths in $[0,f]$ and $(f,1]$ are
\begin{equation}
\begin{aligned}\label{eqn:NL2}
P(m-1, \ell-1) &= \left[qr + q(1-r)\frac{\ell}{m}\right]\mathbf{1}\{ \ell >0 \}, \\
P(m-1,\ell) &= \left[q(1-r)(1-\frac{\ell}{m})\right]\mathbf{1}\{m > 0\} + qr\mathbf{1}\{ m>0, \ell=0 \},
\end{aligned}
\end{equation}
and, finally, if $M_n=0$ and a death event occurs nothing happens: $P(m,\ell) = q \mathbf{1}\{m=0\}$. These transition probabilities are illustrated in Figure~\ref{fig:triangular_grid_chain}. 

\begin{figure}[h!]
\centering
\begin{tikzpicture}[scale=0.92,
    x=5cm, y=3.3cm, 
    state/.style={circle, draw, minimum size=0.2cm, inner sep=0pt, font=\scriptsize},
    arr/.style={->, thick, >=stealth},
    prob/.style={font=\scriptsize, pos=0.5, sloped}
]
\def\Xmax{2} 

\draw[->, opacity=0.2] (-0.3,0) -- (\Xmax+0.6,0) ;
\node[below ] at (\Xmax+0.6,0) {$M_n$};
\draw[->, opacity=0.2] (0,-0.3) -- (0,\Xmax+0.6) ;
\node[ right] at  (0,\Xmax+0.6){$L_n(f)$};

\foreach \x in {0,...,\Xmax}{
  \foreach \y in {0,...,\Xmax}{
    \ifnum\x < \y
    \else
      \node[state] (v\x_\y) at (\x,\y) {$(\x,\y)$};
    \fi
  }
}

\begin{scope}
\clip (-0.35,-0.45) rectangle (\Xmax+0.3,\Xmax+0.3);

\draw[arr] (v0_0) to[out=135, in=225, looseness=5]
  node[prob, left, rotate=90] {$q$} (v0_0);


\foreach \x in {0,...,\numexpr\Xmax-1\relax}{
  \foreach \y in {0,..., \x}{
    \pgfmathtruncatemacro{\xp}{\x+1}
    \draw[arr] (v\x_\y) to[out=20, in=160]
      node[prob, above=0pt, yshift=-2.5pt] {$p(1-f)$} (v\xp_\y);
  }
}

\foreach \x in {1,...,\Xmax}{
  \foreach \y in {0}{
    \pgfmathtruncatemacro{\xm}{\x-1}
    \ifnum \y>\xm\relax
    \else
      \draw[arr] (v\x_\y) to[out=200, in=-20]
        node[prob, above] {$q$} (v\xm_\y);
    \fi
  }
}

\foreach \x in {1,...,\Xmax}{
  \foreach \y in {1,...,\x}{
    \pgfmathtruncatemacro{\xm}{\x-1}
    \pgfmathtruncatemacro{\ym}{\y-1}
    \ifnum\y=\x
      \draw[arr] (v\x_\y) to[bend left=12]
        node[prob, above=0pt] {$q$} (v\xm_\ym);

      \draw[arr] (v\xm_\ym) to[bend left=12]
        node[prob, above=0pt] {$pf$} (v\x_\y);
    \fi
  }
}

\foreach \kk in {1,...,\Xmax}{
  \foreach \xx in {0,...,\Xmax}{
    \pgfmathtruncatemacro{\y}{\xx-\kk}

    \pgfmathtruncatemacro{\xm}{\xx-1}
    \pgfmathtruncatemacro{\ym}{\y-1}

    \pgfmathtruncatemacro{\kshow}{\kk}
    \pgfmathtruncatemacro{\xshow}{\xx}

    \ifnum\y>0\relax
      \ifnum\xx>0\relax
        \draw[arr] (v\xx_\y) to[bend left=12]
          node[prob, above=0pt] {$qr+q(1-r)\dfrac{\y}{\xx}$} (v\xm_\ym);

        \draw[arr] (v\xm_\ym) to[bend left=12]
          node[prob, above=0pt] {$pf$} (v\xx_\y);
      \fi
    \fi
  }
}

\foreach \x in {1,...,\Xmax}{
  \foreach \y in {1,...,\numexpr\x-1\relax}{
    \pgfmathtruncatemacro{\xm}{\x-1}

    \pgfmathtruncatemacro{\xshow}{\x}
    \pgfmathtruncatemacro{\yshow}{\y}
    \pgfmathtruncatemacro{\num}{\x-\y}

    \ifnum\y<\x\relax
      \ifnum\y>0\relax
        \draw[arr] (v\x_\y) to[out=200, in=-20]
          node[prob, above=0pt] {$q(1-r)\dfrac{\num}{\xshow}$} (v\xm_\y);
      \fi
    \fi
  }
}


\pgfmathtruncatemacro{\xb}{\Xmax}
\pgfmathtruncatemacro{\xnext}{\Xmax+1}

\foreach \y in {0,...,\Xmax}{
  \draw[arr] (v\xb_\y) to[out=20, in=160] (\xnext,\y);

  \draw[arr] (\xnext,\y) to[out=200, in=-20] (v\xb_\y);
}

\foreach \y in {0,...,\Xmax}{
  \pgfmathtruncatemacro{\ynext}{\y+1}

  \draw[arr] (v\xb_\y) to[bend left=12] (\xnext,\ynext);

  \draw[arr] (\xnext,\ynext) to[bend left=12] (v\xb_\y);
}

\end{scope}

\foreach \y in {0,...,\Xmax}{
  \node at (\Xmax+0.5, \y) {\huge$\cdots$};
}

\foreach \y in {0,...,\Xmax}{
  \pgfmathsetmacro{\yd}{\Xmax+0.475-\y}
  \node[rotate=35] at (\Xmax+0.475, \yd) {\huge$\cdots$};
}

\end{tikzpicture}
\caption{Transition probabilities of $(M_n,L_n(f))$.}
\label{fig:triangular_grid_chain}
\end{figure}
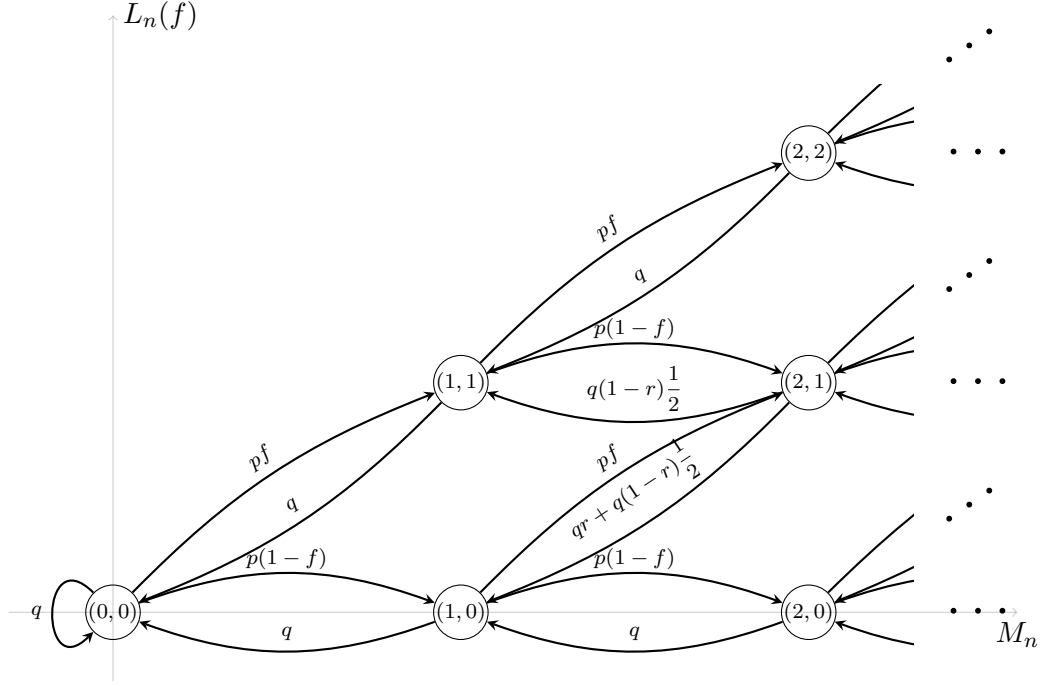

 While $(M_n,L_n(f))$ now satisfies the Markov property, well-known results for 1-D birth–death chains no longer apply. In comparison to \emph{(a)}--\emph{(c)} for the GMS process in Section \ref{sec:GMS_Ln}, we have several additional challenges:
 \begin{itemize}
\item[\emph{(a)}] Recall that $\{M_n : n \in \mathbb{N}_0\}$ \emph{is} a 1-D Markov chain which corresponds to a random walk with a reflecting barrier at zero, which increases by 1 with probability $p$ and decreases by 1 probability $q$. When $p>q$, $M_n$ is transient, and hence $(M_n,L_n(f))$ is also transient. Thus, Theorem \ref{theorem:RPD_L_n} does not boil down to determining the positive-recurrence/null-recurrence/transience of a Markov chain. Moreover, for the RPD process the distribution of $B_{i+1}-B_i$ depends on $i$, whereas for the GMS model it does not.
\item[\emph{(b)}] Unlike for the GMS model, for the RPD process $L_n(f)$ cannot be expressed as a functional of a sum of independent random variables (recall \eqref{eqn:decomp}). This is because the transition probabilities are \emph{state-dependent}. Thus, to prove Theorem \ref{theorem:RPD_Uniformity} we can not rely on results for sums of independent random variables.
\item[\emph{(c)}] Unlike the GMS model, the event a species with fitness $f$ dies does not correspond to the event that $L_n(f)$ eventually hits zero; therefore to prove Theorem \ref{theorem:RPD_survival} a different approach is required.
 \end{itemize}
For these reasons, a fundamentally different proof approach than for the GMS is required for the RPD process. This is approach is detailed in Section \ref{sec:proofs}.

\section{Random location deaths}\label{section:rld_process}

\subsection{Model.}\label{sec:RLDdef} The random location deaths (RLD) process, is characterised by the following dynamics. At time $n = 0$, no species exist and at each time step $n \geq  1$:
\begin{itemize}
    \item With probability $p$, a new species is born and independently assigned a fitness value $f_n \sim $ Uniform$[0,1]$. Otherwise,
    \item with probability $q=1-p$, one species dies, in which case:
          \begin{itemize} 
            \item with probability $r$, the species with lowest fitness value dies. Otherwise,
            \item with probability $1 - r$, sample $U_n \sim \text{Uniform}[0,1]$, and:
            \begin{itemize}
                \item[(*)] If there exists at least one species with fitness less than $U_n$,  
                the species with the \emph{maximum} fitness among those dies.
                \item[(**)] If no species exist with fitness less than $U_n$, the species with the lowest fitness value dies.
            \end{itemize}
  \end{itemize}
\end{itemize}

\subsection{Key results and conjectures.}\label{sec:RLD-KR}
Properties (3-a)--(3-c) of the RLD model that were described in Section \ref{ch:introduction} are formalised in the following theorems and conjectures. The reasoning behind the conjectures is explained in Section \ref{sec:RLD-Ln}.  

\begin{theorem}\label{theorem:RLD_rec}
    Let $f^\ddagger_c = \frac{qr}{p-q(1-r)}$ and suppose $p>q$. 
    \begin{itemize}
        \item[(i)] If $f<f^\ddagger_c$, then $|\mathcal{B}(f)|=\infty$ a.s. and $\mathbb{E}[B_{i+1} - B_i] < \frac{f^\ddagger_c}{f^\ddagger_c-f}$ for all $i$.
        \item[(ii)] If $f=f^\ddagger_c$, then $|\mathcal{B}(f)|=\infty$ a.s.
    \end{itemize}
\end{theorem}

\begin{conjecture}\label{conj:RT}
    Let $f^\ddagger_c = \frac{qr}{p-q(1-r)}$ and suppose $p>q$. 
        \begin{itemize}
        \item[(ii)] If $f=f^\ddagger_c$, then $\mathbb{E}[B_{i+1} - B_i] \to \infty$ as $i \to \infty$.
        \item[(iii)] If $f>f^\ddagger_c$, then $|\mathcal{B}(f)|<\infty$ a.s.
    \end{itemize}
\end{conjecture}

\begin{conjecture}\label{conj:U}
    Let $f^\ddagger_c = \frac{qr}{p-q(1-r)}$ and suppose $p>q$. Then for any $f^\ddagger_c \leq a \leq b \leq 1$,
    \[
    \frac{L_n(b)-L_n(a)}{n} \stackrel{\mathbb{P}}{\to} (p-q) \left( \frac{b-a}{1-f^\ddagger_c} \right), \qquad \text{as } n \to \infty.
    \]
\end{conjecture}

\begin{theorem} 
    Suppose $r<1$ then the probability that any species with fitness $f<1$ survives forever is 0. 
    \label{theorem:RLD_survival}
\end{theorem}

We again note the similarity between these results and those for the GMS and RPD processes in Sections \ref{sec:GMS-KR} and \ref{sec:RPD-KR}, with the key difference again being the formulas for $f_c$. The next result covers the setting of Theorem \ref{theorem:RLD_survival}, in the special case where $f=1$.

\begin{proposition}\label{prop:grp}
If $p>q$ and there are $k$ species alive with fitness less than 1, then the probability that a species with fitness $f=1$ survives forever is $1-(p/q)^{k+1}$.
\end{proposition}

\begin{remark}\label{rem:no_conj}
Theorem \ref{theorem:RLD_rec} provides insight in the behaviour displayed in the top right of Figure \ref{fig:intro_gms_process_behaviour} where $r=1/2$ and $p=0.50025$. When $p=0.50025$,  $M_n$ grows linearly in $n$ (see Lemma \ref{lemma:finding_transient_c} below), i.e., $M_n \approx n(p-q)=0.0005n$, whereas by Theorem \ref{theorem:RLD_rec} we know that the number of species with fitness in $[0,f]$ cannot grow linearly when $f \leq f^\ddagger_c = 0.998$. Thus, the species alive when $n$ is large must have fitnesses that are overwhelmingly in $[0.998,1]$. This is what we observe in Figure \ref{fig:intro_gms_process_behaviour}.
\end{remark}

\subsection{Representation of $L_n(f)$.}\label{sec:RLD-Ln} The difficulty of analysing this process is due to the transition probabilities of $L_n(f)$. Let 
\[
R_{\min}:= \inf \{ g \in \mathcal{M}_n : g \geq f\} \wedge 1
\]
denote the species with the smallest fitness greater than $f$. Suppose $L_n(f)>0$, then the transition probabilities for $L_n(f)$ are determined by the following (disjoint) events:
\begin{itemize}
    \item A species is born with fitness in $[0,f]$ (i.e., $L_n(f)$ increases by 1) with probability $pf$.
    \item A death happens that target the least-fit (i.e., $L_n(f)$ decreases by 1) with probability $qr$.
    \item A random–location death occurs and targets a species within $[0,f]$ (i.e., $L_n(f)$ decreases by 1) with probability $q(1-r)R_{\min}$.
\end{itemize}
These transition probabilities are depicted in Figure~\ref{fig:L_n^*(f)_RLD}. Note that the dependence on $R_{\min}$ means that this is not a Markov chain. Moreover, extending the state space to two dimensions, as we did for the RPD process, is insufficient: while the pair \((R_{\min},L_n(f))\) does encode the information needed to determine the next transition probabilities $L_n(f)$, it does not encode enough information for those of $R_{\min}$. 
\begin{figure}[h!]
    \centering
    \begin{tikzpicture}[
        node distance=4.5cm,
        state/.style={circle, draw, minimum size=1cm}
    ]
        \node[state] (A0) at (0,0) {0};
        \node[state, right=of A0] (A1) {1};
        \node[state, right=of A1] (A2) {2};
        \begin{scope}
            \clip (-1.5,-2.5) rectangle ($(A2)+(2.2,3.5)$);
            \draw[->, thick] (A0) to[out=30, in=150] node[above] {$pf$} (A1);
            \draw[->, thick] (A1) to[out=30, in=150] node[above] {$pf$} (A2);
            \draw[->, thick] (A2) to[out=30, in=150] 
                node[midway, above] {} ($(A2)+(4.5,0.25)$);
            \draw[->, thick] (A1) to[out=210, in=330] node[below] {$qr+q(1{-}r)R_{\min}$} (A0);
            \draw[->, thick] (A2) to[out=210, in=330] node[below] {$qr+q(1{-}r)R_{\min}$} (A1);
            \draw[->, thick] ($(A2)+(4.5,-0.25)$) 
                to[out=210, in=330] node[midway, below] {} (A2);
            \draw[->, thick] (A0) edge[loop above, looseness=15] node {$1-pf$} (A0);
            \draw[->, thick] (A1) edge[loop above, looseness=15] node {$1-pf-qr-q(1{-}r)R_{\min}$} (A1);
            \draw[->, thick] (A2) edge[loop above, looseness=15] node {$1-pf-qr-q(1{-}r)R_{\min}$} (A2);
        \end{scope}
        \node at ($(A2)+(2.8,0)$) {\dots};
    \end{tikzpicture}
    \caption{Transition probabilities of $L_n(f)$ in the RLD process.}
    \label{fig:L_n^*(f)_RLD}
\end{figure}
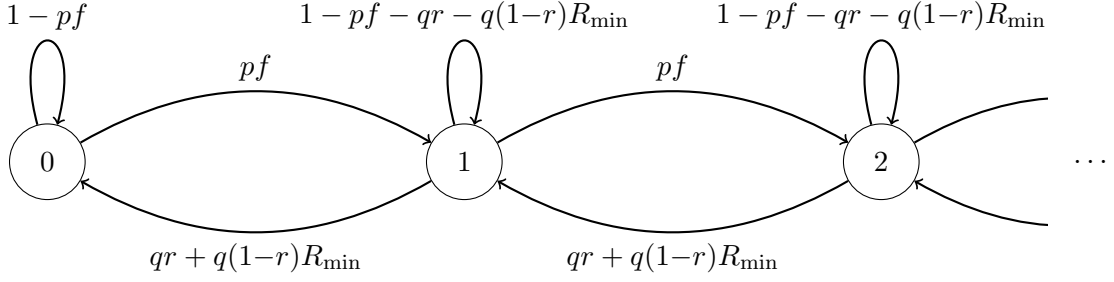

The reasoning behind Conjectures \ref{conj:RT} and \ref{conj:U} is the following. When $p>q$ and $f\geq f^\ddagger_c$, we expect more and more species to accumulate in the interval $[f,1]$ as $n \to \infty$. Because $R_{\min}$ is the species with the smallest fitness in $[f,1]$, we thus expect $R_{\min} \stackrel{\mathbb{P}}{\to} f$ as $n \to \infty$. Consequently, for large $n$ and a specific $f \geq f^\ddagger_c$, we believe that the behaviour of $L_n(f)$ is asymptotically well approximated by the Markov chain depicted in Figure \ref{fig:L_n^*(f)_RLD} \emph{with $R_{\min}$ replaced by $f$}. With this modification, the Markov chain is random walk with an absorbing barrier at $0$, and hence the arguments described at the end of Section \ref{sec:GMS_Ln} can be applied. Applying these arguments to this modified Markov chain yields Conjectures \ref{conj:RT} and \ref{conj:U}.

\section{Discussion}\label{sec:discussion}

By considering the GMS, RPD, and RLD processes together, we are able to compare them directly. Consider the following general definition for $f_c$:
\[
f_c = \sup \{ f: |\mathcal{B}(f)| = \infty \text{ almost surely} \} \wedge 1,
\]
where we take the convention that $f_c=1$ when there is no value of $f$ such that $|\mathcal{B}(f)| < \infty$ with positive probability. Recall that in all three models $M_n \equiv L_n(1)$ is random walk with a reflecting barrier at 0 which increases by 1 with probability $p$ and decreases by 1 with probability $q$. For all three models we therefore have $f_c=1$ when $p \leq 1/2$ and $f_c<1$ when $p>1/2$. When $p>1/2$ we proved that $f_c=f^*_c=\frac{q}{p}$ for the GMS process and $f_c=f^\dagger_c=\frac{qr}{p}$ for the RPD process, whereas we conjectured that $f_c=f^\ddagger_c=\frac{qr}{p-q(1-r)}$ for the RLD process. In Figure \ref{fig:comparing_three_models} we plot $f_c$ for the three models when $r=0.5$ assuming this conjecture is correct. Notice that $f_c$ is continuous in $p$ for the GMS and RLD processes, but discontinuous at $p=1/2$ for the RPD process. It is this discontinuity that leads to the contrasting behaviour in Figure \ref{fig:intro_gms_process_behaviour}. We note that Theorem \ref{theorem:RLD_rec} implies $f_c \geq f^\ddagger_c$ for the RLD process and hence we proved that $f_c$ must be continuous at $p=1/2$ for this model. 

\begin{figure}[h!]
    \centering
    \includegraphics[width=0.9\linewidth]{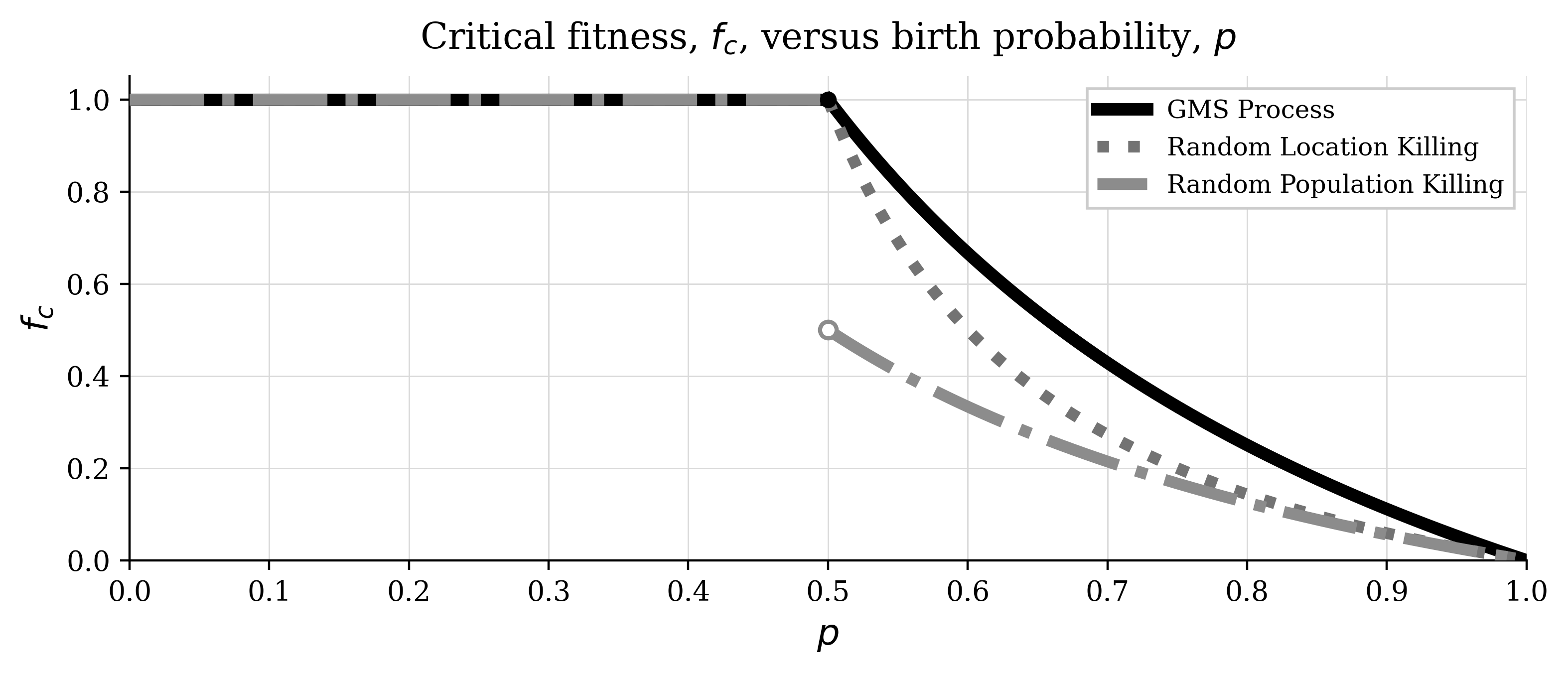}
    \caption{The critical fitness function $f_c$ as a function $p$ for the three models}
    \label{fig:comparing_three_models}
\end{figure}

It is curious that the discontinuity appears for the RPD process but not for the RLD process despite both having similar random death mechanisms. We now offer an intuitive explanation:
\begin{itemize}
\item In the RPD process, a random death event targets each species with equal probability; consequently, a species with fitness 0.999 has \emph{effectively no competitive advantage} over a species with fitness 0.99. Species with exceptionally high fitness are always in danger from a random death event and therefore, when $p \approx q$, they do not accumulate faster than other species with a high fitness. 
\item In the RLD process, a random death event involves simulating a $U_n \sim \text{Uniform}[0,1]$ random variable and killing the species with the maximal fitness less than $U_n$. If no species has fitness less than $U_n$, then the species with the lowest fitness is killed. Suppose the population consists two species, one with fitness 0.999 and another with fitness 0.99, then a random death event kills the species with fitness 0.99 with probability 0.999, and the species with fitness 0.999 with probability 0.001. In other words, species with exceptionally high fitness can effectively ``hide'' behind species with a high fitness (whereas in the RPD proess they cannot). Consequently, species with exceptionally high fitness in the RLD process \emph{do have a competitive advantage} over species with a high fitness. Therefore, when $p\approx q$, species with exceptionally high fitness accumulate faster than other species with a high fitness.
\end{itemize}

Unlike the RPD process, the RLD process similar to the GMS in that they both have continuous $f_c$ functions at $p=1/2$ (Theorems \ref{theorem:gms_theorem}, \ref{theorem:RPD_L_n} and \ref{theorem:RLD_rec}). In contrast, unlike the GMS process, the RLD process is similar to the RPD process in that species with fitness $f<1$ cannot survive indefinitely (Theorems \ref{theorem:gms_survival}, \ref{theorem:RPD_survival} and \ref{theorem:RLD_survival}). Thus, the RLD process can be viewed as being in-between the GMS and RPD processes. This can been seen in Table \ref{tab:comparison-processes} where we summarise the key properties of the processes.

\begin{table}[h!]
\centering
\renewcommand{\arraystretch}{1}
\setlength{\tabcolsep}{15pt}
\begin{tabularx}{0.9\textwidth}{l c c c}
\toprule
\textbf{Process} &
\shortstack[c]{Continuous $f_c(p)$\\ at $p=1/2$} &
\shortstack[c]{Fitness distribution\\[+0.2em] uniform on $[f_c,1]$} &
\shortstack[c]{Species survive\\ forever w.p.p.} \\
\midrule
\textbf{GMS}  & \textcolor{black}{\cmark} & \textcolor{black}{\cmark} & \textcolor{black}{\cmark} \\
\addlinespace[0.4em]
\textbf{RPD}  & \textcolor{black}{\xmark} & \textcolor{black}{\cmark} & \textcolor{black}{\xmark} \\
\addlinespace[0.4em]
\textbf{RLD}  & \textcolor{black}{\cmark} & \textcolor{black!35}{\cmark} & \textcolor{black}{\xmark} \\
\bottomrule
\end{tabularx}
\caption{Established (black) and conjectured (grey) properties of the GMS, RPD and RLD processes.}
\label{tab:comparison-processes}
\end{table}

Our results motivate several future directions of research. It still remains to prove Conjectures \ref{conj:RT} and \ref{conj:U}. In addition, we believe the results in Section \ref{sec:RLD-KR} hold for modifications of the model introduced in Section \ref{sec:RLDdef}; for example, if (*) in Section \ref{sec:RLDdef} is replaced by: If there exists at least one species with fitness \textit{greater than} $U_n$, the species with the \textit{minimum} fitness among those dies. More broadly, we would like to understand if the discontinuity in $f_c$ for the RPD process persists if the birth mechanism of the model is updated to more accurately reflect nature. For example, if the fitness of a new species is chosen so that it is similar to that of an existing species.

\section{Proofs}\label{sec:proofs}

\subsection{Proofs for the RPD process}

\begin{proof}[Proof of Theorem \ref{theorem:RPD_L_n}(i)]
Suppose $f<f^\dagger_c$. We introduce Markov chain labelled $D_n(f)$ with state-space $\mathbb{N}_0$ which is designed to dominate $L_n(f)$. Letting $P(d'):=\mathbb{P}(D_{n+1}(f)=d' \mid D_n(f) = d)$, this Markov chain is characterised by the transition probabilities 
\begin{equation}\label{eqn:D_n}
P(d-1)=qr, \quad P(d)=1-qr-pf \quad \text{and} \quad P(d+1)=pf
\end{equation}
when $d>0$, and $P(d)=1-pf$ and $P(d+1)=pf$ when $d=0$. Note that these are the transition probabilities we obtain for $L_n(f)$ in the RPD process when $M_n \to \infty$. 

We couple $D_n(f)$ and the Markov chain $(M_n,L_n(f))$ for the RPD process, by defining a joint dependence between them while maintaining the same marginal transition probabilities for the two processes. In other words, we define a new Markov chain $(D_n(f), M_n, L_n(f)$) such that, when considered separately, $D_n(f)$ has transition probabilities given by \eqref{eqn:D_n}, and $(M_n,L_n(f))$ has transition probabilities given by \eqref{eqn:NL1} and \eqref{eqn:NL2}. 
We suppose the initial state is $(D_0(f), M_0, L_0(f)) = (0, m_0, 0)$ for $m_0 \in \mathbb{N}_0$.
Let
\[
P(d',m',\ell'):=\mathbb{P}\bigg(D_{n+1}(f)=d',M_{n+1}=m', L_{n+1}(f)=\ell' \; \bigg| \; D_n(f) = d,M_{n}=m, L_{n}(f)=\ell \bigg).
\]
We let the transition probabilities that correspond to births be
\[
P(d+1,m+1, \ell+1) = pf, \qquad P(d,m+1, \ell) = p(1-f)
\]
When we define the transition probabilities that correspond to deaths there are several cases due to boundary conditions (i.e., because the number of species must always be non-negative):
\smallskip
\noindent 
\underline{Case 1:} $d,m,\ell>0$
\[
 P(d-1,m-1, \ell-1) = qr, \qquad P(d,m-1, \ell-1) = q(1-r)\frac{\ell}{m}, \qquad P(d,m-1,\ell) = q(1-r)(1-\frac{\ell}{m}).
\]

\noindent 
\underline{Case 2:} $d,m>0$, $\ell=0$
\[
 P(d-1,m-1, 0) = qr, \qquad  P(d,m-1, 0) = q(1-r)
\]

\noindent 
\underline{Case 3:} $d>0$, $m=\ell=0$
\[
 P(d-1,0, 0) = qr, \qquad  P(d,0, 0) = q(1-r)
\]

\noindent 
\underline{Case 4:} $d=0$, $m>0$, $\ell=0$
\[
P(0,m-1, 0) = q
\]

\noindent 
\underline{Case 5:} $d=0$, $m=0$, $\ell=0$
\[
P(0,0, 0) = q
\]
We can directly verify that $D_n(f)$, and $(M_n,L_n(f))$ have the transition probabilities in \eqref{eqn:D_n} and, \eqref{eqn:NL1} and \eqref{eqn:NL2} respectively. In addition, excluding boundary cases 3 and 4 where $d>0$ and $\ell=0$, in ($D_n(f)$, $M_n$, $L_n(f)$), the increment of $D_n(f)$ (i.e. $D_{n+1}(f)-D_n(f))$ is always greater than or equal to the increment of $L_n(f)$. Consequently, for any $m_0\geq 0$, given $(D_n(f), M_n, L_n(f))=(0,m_0,0)$, we have $D_n(f) \geq L_n(f)$ with probability 1 for all $n \geq 0$.

To establish the result, first observe that by applying the Markov property for the chain $(M_n,L_n(f))$ in the second step we have
\begin{align}
\mathbb{E}(B_{i+1}-B_i) &= \sum_{m_0=0}^\infty \mathbb{E}(B_{i+1} - B_i \mid M_{B_i}=m_0)\mathbb{P}(M_{B_i}=m_0) \nonumber \\
&\leq \sup_{m_0 \geq 0} \mathbb{E}(B_1-B_0 \mid M_0=m_0). \label{eqn:thm3i1}
\end{align}
Next, let 
\begin{equation}\label{eq:Adef}
    \mathcal{A}(f)=\{n \geq 0 : D_n(f)=0\}
\end{equation} and denote the elements of $\mathcal{A}(f)$ by $0=A_0< A_1<A_2 < \dots$. 
Recall that for any $M_0=m_0$, $D_n(f) \geq L_n(f)$ for all $n \geq 0$. Consequently if $D_n(f) =0$ then $L_n(f)=0$, which implies $\mathcal{A}(f) \subseteq \mathcal{B}(f)$, and therefore $B_1 \leq A_1$ with probability 1. Given $A_0=B_0=0$, this means
\begin{equation}\label{eqn:thm3i2}
\sup_{m_0 \geq 0} \mathbb{E}(B_1-B_0 \mid M_0=m_0) \leq \mathbb{E}(A_1-A_0 ).
\end{equation}
Finally, recall that $D_n(f)$ is a simple birth-death Markov chain with transition probabilities given by \eqref{eqn:D_n}. Note that $f<f^\dagger_c$ implies that $pf<qr$, and hence $D_n(f)$ is positive recurrent and therefore has a stationary distribution. Applying the standard formula, we obtain that the stationary probability that the chain $D_n(f)$ is in state 0 is $1-\frac{pf}{qr} = \frac{f^\dagger_c-f}{f^\dagger_c}$. Recalling that the stationary probability of a state in a Markov chain is the reciprocal of the return time, we obtain
\begin{equation}\label{eqn:thm3i3}
\mathbb{E}(A_1-A_0 ) = \frac{f^\dagger_c}{f^\dagger_c-f}.
\end{equation}
Combining \eqref{eqn:thm3i1}, \eqref{eqn:thm3i2}, and \eqref{eqn:thm3i3} then yields the result.
\end{proof}

\begin{proof}[Proof of Theorem \ref{theorem:RPD_L_n}(ii)]
Suppose $f=f^\dagger_c=\frac{qr}{p}$. Recall the coupled chain $(D_n(f),M_n, L_n(f))$ from the proof of Theorem \ref{theorem:RPD_L_n}(i). When $f=f^\dagger_c$ the step forward probability for $D_n(f)$ in \eqref{eqn:D_n} is $pf$ which is equal to the step back probability $qr$, and hence $D_n(f)$ is null-recurrent. Recalling the definition of $\mathcal{A}(f)$ from \eqref{eq:Adef}, we then have $|\mathcal{A}(f)|=\infty$ a.s. and $\mathbb{E}(A_1-A_0)=\infty$. As in the proof of Theorem \ref{theorem:RPD_L_n}(i), we have $\mathcal{A}(f) \subseteq \mathcal{B}(f)$ and therefore $|\mathcal{B}(f)|=\infty$ a.s., thereby proving the first assertion.

To show that $\mathbb{E}(B_{i+1}-B_i) \to \infty$ as $i \to \infty$, we note that $M_n \to \infty$ a.s. as $n \to \infty$. Indeed, this follows from the stronger result that $\frac{M_n}{n}\to p-q$ as $n \to \infty$ a.s. (see Lemma \ref{lemma:finding_transient_c} below). Consequently, using the Markov property in this first step, if we can show that the second equality holds in
\begin{equation}
\lim_{m_0\to \infty} \mathbb{E}(B_{i+1} - B_i \mid M_{B_i}=m_0) = \lim_{m_0 \to \infty} \mathbb{E}(B_{1} - B_0 \mid M_{0}=m_0) = \infty.
\end{equation}
then this implies the result.
Recalling that $D_0(f)=L_0(f)=0$, let $T = \inf \{ n : D_n(f) \neq L_n(f)\}$. 
Given the transition probabilities of $(D_n(f), M_n, L_n(f))$ described above in the proof of Theorem \ref{theorem:RPD_L_n}(i) and that $M_0=m_0$, if we suppose that $m_0>n$ then we have that $\mathbb{P}(T=n+1 \mid T\geq n) \leq \frac{n}{m_0-n}$ independently of the history of the process. 
This is because if $T\geq n$, then $L_{n}(f)=D_n(f)$ and it is only in Case 1 in the coupling described above where it is possible that $L_{n+1}(f)\neq D_{n+1}(f)$ and this occurs with probability $q(1-r)L_{n}(f)/M_n$. More rigorously, using the law of total probability in the first step and the fact that at most one species can be born/killed at each time step in the second, we have
\begin{align*}
\mathbb{P}(T=n+1 &\mid T \geq n, \, M_0=m_0) \\ 
&= \sum_{\ell=0}^\infty \sum_{m'=0}^\infty q(1-r)\frac{\ell}{m'} \mathbb{P}(L_n(f)=\ell, \, M_n=m' \mid T >n, \, M_0=m_0) \\
&=\sum_{\ell=0}^n \sum_{m'=m_0-n}^{m_0+n} q(1-r)\frac{\ell}{m'} \mathbb{P}(L_n(f)=\ell, \, M_n=m' \mid T \geq n, \, M_0=m_0) \\
&\leq q(1-r) \frac{n}{m_0-n}.
\end{align*}
Consequently, for any fixed $c < m_0$, we have 
\begin{align*}
\mathbb{P}(T \leq c \mid M_0 = m_0) &= 1 - \mathbb{P}(T > c \mid M_0 = m_0) \\
&= 1 - \prod_{n=0}^{c-1} (1- \mathbb{P}(T=n+1\mid T \geq n, \, M_0=m_0)) \\
&\leq 1-\left(1 - q(1-r)\frac{c}{m_0-c}\right)^c,
\end{align*}
which converges to 0 as $m_0 \to \infty$. 
We thus have
\begin{align*}
\lim_{m_0 \to \infty} \mathbb{E}(B_{1}-B_0 \mid M_n=m_0) &\geq \lim_{m_0 \to \infty} \mathbb{E}(\min\{A_{1}-A_0, \, T\} \mid M_0=m_0) \\
&\geq \lim_{c \to \infty} \lim_{m_0 \to \infty} \mathbb{E}(\min\{ A_{1}-A_0, \, c\}\mathbf{1}\{ T>c\} \mid M_0 = m_0) \\
&\geq \lim_{c \to \infty} \lim_{m_0 \to \infty} \mathbb{E}(\min\{ A_{1}-A_0, \, c\} - c\mathbf{1}\{T \leq c\} \mid M_0 = m_0) \\
&\geq \lim_{c \to \infty} \mathbb{E}(\min\{ A_{1}-A_0, \, c\}) -  \lim_{c \to \infty} \lim_{m_0 \to \infty}  \mathbb{E}( c\mathbf{1}\{T \leq c\} \mid M_0 = m_0) \\
&=\lim_{c \to \infty}  \mathbb{E}(\min\{ A_{1}-A_0, \, c\}) \\
&= \mathbb{E}(A_1 - A_0) \\
&= \infty,
\end{align*}
where in the second last step we apply the monotone convergence convergence theorem and in the last we recall that $\{D_n: n \in \mathbb{N}_0\}$ is null-recurrent when $f=f^\dagger_c$.
\end{proof}

We defer the proof of Theorem \ref{theorem:RPD_L_n}(iii) until after the proof of Theorem \ref{theorem:RPD_Uniformity}. This is because the proof of Theorem \ref{theorem:RPD_L_n}(iii) uses some of the results developed in the proof of Theorem \ref{theorem:RPD_Uniformity}. 

We break the proof of Theorem \ref{theorem:RPD_Uniformity} into several steps. Before beginning we provide an overview of the proof strategy. First note that by Slutsky's theorem it is sufficient to establish that 
\begin{equation}\label{eq:thm32gl}
\frac{L_n(f)}{n} \stackrel{\mathbb{P}}{\to}(p-q) \left( \frac{f-f^\dagger_c}{1-f^\dagger_c} \right)
\end{equation}
for $f \in (f^\dagger_c,1]$ as $n \to \infty$.
Rather than establishing \eqref{eq:thm32gl} directly, we instead proceed by establishing a functional law of large numbers for the process $(\mathbf{X}^n(t))_{t\in[0,1]}$ where
\[
\mathbf{X}^n(t) = \begin{bmatrix}
    X^n_1(t) \\[10pt]
    X^n_2(t)
\end{bmatrix} := \begin{bmatrix}
    \dfrac{M_{\lfloor nt \rfloor}}{n} \\[10pt]
    \dfrac{L_{\lfloor nt \rfloor}(f)}{n}
\end{bmatrix}, 
\]
as $n \to \infty$. As a consequence, we obtain a (uni-variate) law of large numbers for $\frac{L_n(f)}{n}$ by setting $t=1$, and ignoring the component which concerns $M_n$. This approach helps us to deal with the fact that $(M_n, L_n(f))$ is a density–dependent Markov chain, i.e. its transition probabilities (illustrated in Figure~\ref{fig:triangular_grid_chain}) depend heavily on the current state of the process. 
We can establish the functional law of large numbers by applying a general result from Kurtz \cite[Theorem 4.7]{FluidLimitKurtz}. However, to apply this result, the transition probabilities of the Markov chain ($M_n,L_n(f)$) must satisfy certain Lipschitz continuity conditions. It turns out that these Lipschitz continuity conditions are not satisfied on the boundary where $L_{n}(f) = 0$ (due to the term $qr\mathbf{1}\{m >0, \ell=0\}$ in \eqref{eqn:NL2}). This leads us to split the interval $t\in[0,1]$ into two parts: (1) $t \in [0, \epsilon]$, where our focus is on showing that for any $\epsilon>0$ there exists $\eta(\epsilon)>0$ such that $\frac{L_{\lfloor nt \rfloor}(f)}{n} > \eta(\epsilon)$ for some $t \in [0,\varepsilon]$ with high probability, and (2) $t\in [\epsilon,1]$, where we are then able to apply \cite[Theorem 4.7]{FluidLimitKurtz} because we no longer need to be concerned about the $L_n(f)=0$ boundary (and hence we are able to ignore the term $qr\mathbf{1}\{n >0, \ell=0\}$ in \eqref{eqn:NL2}). The result then follows by taking the resulting fluid limit $\mathbf{X}^n \Rightarrow \mathbf{x}$ which is expressed as a differential equation, and showing that when $\epsilon \downarrow 0$ we obtain $\mathbf{x}(1) \to \left[p-q, (p-q)\frac{f-f^\dagger_c}{1-f^\dagger_c}\right]^\top$. With our general approach in mind, the proof is structured as follows:
\begin{enumerate}
\item Show the process leaves the $\frac{L_{\lfloor nt \rfloor}(f)}{n}=0$ boundary during $t \in [0,\varepsilon]$ (Lemmas \ref{lemma:finding_transient_c} and \ref{lem:0eps}).
\item Establish a functional law of large numbers for $(\mathbf{X}^n(t))_{t \in [\varepsilon,1]}$ (Lemma \ref{lem:fluid}).
\item Use Steps 1 and 2 to prove Theorem \ref{theorem:RPD_Uniformity}.
\end{enumerate}

\begin{lemma}
For $p>q$,
    \[
    \lim_{n\to\infty} \frac{M_n}{n} = p-q \quad\quad\text{a.s.}
    \]
    \label{lemma:finding_transient_c}
\end{lemma}
\begin{proof}
Note that $M_n$ has the same distribution in both the RPD and GMS processes. Further, recall that $M_n \equiv L_n(1)$. We can therefore apply the decomposition \ref{eqn:decomp}. By the strong law of large numbers we have $\frac{S_n}{n} \stackrel{a.s.}{\to} p-q$. This further implies that, with probability 1, there exists $N <\infty$ such that $S_k \geq 0$ for all $k \geq N$. Consequently, $\frac{\inf_{k \leq n} S_k}{n} \stackrel{a.s.}{\to} 0$. Combining these two facts then gives the result.
\end{proof}

\begin{lemma}\label{lem:0eps}
For the RPD process, if $f>f^\dagger_c$ then $\limsup_{n \to \infty} \frac{L_n(f)}{n} >0$ almost surely.
\end{lemma}
\begin{proof}
Let $f>f^\dagger_c$. We establish the result by contradiction. Suppose that 
\begin{equation}\label{eqn:contass}
\lim_{n \to\infty} \frac{L_n(f)}{n} = 0 \qquad \text{with positive probability.}
\end{equation}
Let
\begin{align*}
TB_n(f) &= \text{ the total number of species born in $[0,f]$ up to time $n$} \\
TLF_n &= \text{ the total number of least fit kills up to time $n$} \\
TLF_n(f) &= \text{ the total number of least fit kills that kill a species in $[0,f]$ up to time $n$} \\
TR_n(f) &= \text{ the total number of random kills that kill a species in $[0,f]$ up to time $n$} 
\end{align*}
By the strong law of large numbers we have
\begin{equation}\label{eq:nTPr}
\frac{TB_n(f)}{n} \stackrel{a.s.}{\to} pf, \qquad \text{and} \qquad \frac{TLF_n}{n} \stackrel{a.s.}{\to} qr \qquad \text{as } n \to \infty.
\end{equation}
Observing that the number of species in $[0,f]$ at time $n$ is the number born in $[0,f]$ minus the number killed in $[0,f]$ we have
\begin{align}
L_n(f) &= TB_n(f) - TLF_n(f) - TR_n(f) \nonumber \\
&\geq TB_n(f) - TLF_n - TR_n(f) , \label{eq:lntb}
\end{align}
where we obtain the inequality by noting that the $TLF_n \geq TLF_n(f)$ by definition. Combining \eqref{eqn:contass}, \eqref{eq:nTPr}, and \eqref{eq:lntb} we conclude that
\begin{equation}\label{eq:lntr2}
\text{If} \quad  \lim_{n \to\infty} \frac{L_n(f)}{n} = 0 \quad \text{then} \quad \liminf_{n \to \infty}\frac{TR_n(f)}{n} \geq pf - qr \quad \text{almost surely.}
\end{equation}
Observe in addition that $f>f^\dagger_c$ is equivalent to $pf-qr>0$. Applying Lemma \ref{lemma:finding_transient_c} and Equation \eqref{eq:lntr2} gives
\begin{align}
    \mathbb{P}\left( \lim_{n \to\infty} \frac{L_n(f)}{n} = 0\right) &= \mathbb{P}\left( \lim_{n \to\infty} \frac{L_n(f)}{n} = 0, \; \liminf_{n \to \infty}\frac{TR_n(f)}{n} \geq pf - qr, \; \lim_{n \to \infty} \frac{M_n}{n}=p-q \right) \nonumber \\
    &=\mathbb{P}\left(\left. \liminf_{n \to \infty}\frac{TR_n(f)}{n} \geq pf - qr \; \right| \;  \lim_{n \to\infty} \frac{L_n(f)}{n} = 0,  \; \lim_{n \to \infty} \frac{M_n}{n}=p-q \right) \nonumber \\
    &\qquad \times \mathbb{P}\left( \lim_{n \to\infty} \frac{L_n(f)}{n} = 0, \; \lim_{n \to \infty} \frac{M_n}{n}=p-q \right) \label{eq:nt0}
\end{align}
To establish the contradiction to \eqref{eqn:contass} and complete the proof we need to show that the right-hand-side of \eqref{eq:nt0} is 0.
Let $I_k:=\mathbf{1}\{\text{at time $k$ there is a random kill that kills a species in $[0,f]$}\}$ so that $TR_n(f)=\sum_{k=1}^n I_k$. Conditional on the history of the process up to time $k-1$ (call this $\mathcal{F}_k$) the probability of a random kill in $[0,f]$ at time $k$ is $q \frac{L_k(f)}{M_k}$, that is,
\[
\mathbb{E}(I_k \mid \mathcal{F}_{k-1}) = q \frac{L_k(f)}{M_k}.
\]
On the event $\{L_n(f)/n \to 0\} \cap \{ M_n/n \to p-q\}$ we have $L_k(f)/M_k \to 0$, and hence 
\begin{equation}\label{eq:Ikmean}
\frac{1}{n}\sum_{k=1}^{n} \mathbb{E}(I_k \mid \mathcal{F}_{k-1}) \to 0
\end{equation}
on this same event. Considering the same event again, applying the martingale law of large numbers we have
\[
\frac{1}{n}\sum_{k=1}^{n} [I_k - \mathbb{E}(I_k \mid \mathcal{F}_{k-1})] \to 0\quad a.s.
\]
as $n \to \infty$. Given \eqref{eq:Ikmean} this implies that 
\[
\lim_{n \to \infty} \frac{TR_n(f)}{n} = \lim_{n \to \infty} \frac{1}{n} \sum_{k=1}^n I_k = 0 \quad a.s.
\]
 as $n \to \infty$ on $\{L_n(f)/n \to 0\} \cap \{ M_n/n \to p-q\}$. Hence the right-hand-side of \eqref{eq:nt0} is 0, which establishes the contradiction to \eqref{eqn:contass}.  
\end{proof}

\begin{lemma}\label{lem:fluid}
Suppose $f > f^\dagger_c$. If $\mathbf{X}^n(0) \to (m_0,\ell_0)^\top$, where $m_0 \geq \ell_0>0$, then for every $\delta,t>0$ 
\begin{equation}\label{eqn:FLLg}
\lim_{n\to\infty}\ \mathbb{P}\!\left\{\ \sup_{0\le s\le t}\ \|\mathbf{X}^n(s)-\mathbf{x}(s)\|>\delta\ \right\}=0,
\end{equation}
where
\begin{equation}\label{eq:Xdeflem}
\mathbf{x}(s) = \begin{bmatrix}
m_0 + (p-q)s \\[+1mm]
\ell_0 \left( \frac{m_0}{m_0 + (p-q)s} \right)^{\alpha} + \left(\frac{f-f^\dagger_c}{1-f^\dagger_c}\right) \left( m_0 + (p-q)s - \frac{m_0^{\alpha +1}}{(m_0 + (p-q)s)^{\alpha}} \right)
\end{bmatrix}
\end{equation}
and $\alpha=\frac{q(1-r)}{p-q}$.
\end{lemma}
\begin{proof}
First observe that for any $s \in \mathbb{R}$ and $n \in \mathbb{N}$, $\mathbf{X}^n(s)$ is a random variable on $\mathcal{S} = \{(m,\ell) \in \mathbb{R}^{2}: m \geq \ell \geq 0\}$.
Following \cite{FluidLimitKurtz}, define the drift function, $\mathbf{f}_n(\mathbf{x})$, where $\mathbf{x}$ is a state in the state space,
\begin{align*}
    \mathbf{f}_n(\mathbf{x}) &:= n \cdot \mathbb{E}[\mathbf{X}^n(\frac{1}{n}) - \mathbf{X}^n(0) | \mathbf{X}^n(0) = \mathbf{x}].
\end{align*}
Using the transition probabilities described in \eqref{eqn:NL1} and \eqref{eqn:NL2}, and the notation $\mathbf{x} = (m_0, \ell_0)^\top$, we obtain 
\[
\mathbf{f}_n(\mathbf{x}) = n\begin{bmatrix} 
\frac{1}{n} p - \frac{1}{n}q \mathbf{1}\{m_0>0\} \\
\frac{1}{n} pf - \frac{1}{n} q(r + (1-r)\frac{\ell_0}{m_0})\mathbf{1}\{\ell_0>0\}
\end{bmatrix} 
= \begin{bmatrix} 
 p - q \mathbf{1}\{m_0>0\} \\
pf - q(r + (1-r)\frac{\ell_0}{m_0})\mathbf{1}\{\ell_0>0\}
\end{bmatrix} =:\mathbf{f}(\mathbf{x})
\]
If we can apply \cite[Theorem 4.7]{FluidLimitKurtz} then this implies that \eqref{eqn:FLLg} holds with $\mathbf{\mathbf{x}}(s)$ replaced by $\mathbf{x}^*(s)$, where $\mathbf{x}^*(s)$ is characterised by 
\begin{equation}\label{eqn:Xstdef}
\mathbf{x}^*(0)=(m_0,\ell_0) \qquad \text{and} \qquad \frac{\partial}{\partial s} \mathbf{x}^*(s) = \mathbf{f}(\mathbf{x}^*(s)).
\end{equation}
We therefore have two steps to do: {Step 1}, determine if/when the conditions of \cite[Theorem 4.7]{FluidLimitKurtz} hold, and {Step 2},  show that $\mathbf{x}^*(s)$ and $\mathbf{x}(s)$ defined in \eqref{eq:Xdeflem} are equivalent.

{Step 1:} To apply \cite[Theorem 4.7]{FluidLimitKurtz} we require that there exists a set $\mathcal{K} \subseteq \mathcal{S}$  such that: (i) that $\lim_{n\to\infty}\ \sup_{\mathbf{x}\in \mathcal{K}}\ \|\mathbf{f}_n(\mathbf{x})-\mathbf{f}(\mathbf{x})\|\;=\;0$ for all $\mathbf{x} \in \mathcal{K}$, which is immediately satisfied because $\mathbf{f}_n(\mathbf{x})$ does not depend on $n$; (ii) that the conditions of \cite[Corollary 2.2]{FluidLimitKurtz} hold, which are effectively bounds on the probability that $\mathbf{X}^n(s)$ has large single-step jumps and are immediately satisfied by observing that only one death/birth can happen at each time step; (iii) that a open ball around the $(\mathbf{x}^*(s))_{0 \leq s \leq t}$ belongs to $\mathcal{K}$, i.e., there exists $\eta >0$ such that $\{ \mathbf{y} \in \mathbb{R}^2 : \inf_{s \leq t} |\mathbf{y} - \mathbf{x}(s)| \leq \eta\} \subset \mathcal{K}$ (see \cite[Theorem 2.11]{FluidLimitKurtz}), and (iv) a Lipshitz continuity condition:
there exists $M>0$ such that
\[
\|\mathbf{f}(\mathbf{x})-\mathbf{f}(\mathbf{y})\|\;\le\;M\,\|\mathbf{x}-\mathbf{y}\|\qquad\text{for all }\mathbf{x},\mathbf{y}\in \mathcal{K}.
\]

Observe that if we take $\mathcal{K} = \mathcal{S}$ then the Lipshitz continuity condition is not satisfied due to the indicators in the expression for $\mathbf{f}(\mathbf{x})$. We therefore let $\mathcal{K} := \{(m,\ell) \in \mathbb{R}^{2}: m \geq \ell \geq \ell_0/2\}$.
In this case, we can take the Lipshitz constant to be $M=2\sqrt{2}/\ell_0<\infty$, and hence (iv) is satisfied. It then remains to verify (iii); however (iii) can be readily verified once we obtain the explicit formula for $\mathbf{x}^*(s)$ in Step 2.

{Step 2:} 
For $\mathbf{x}^*(s)=(m(s),\ell(s))^\top \in \mathcal{K}$, \eqref{eqn:Xstdef} simplifies to
\begin{align*}
       \mathbf{x}^*(0)  = \begin{bmatrix} m_0 \\ \ell_0 \end{bmatrix} 
       \qquad \text{and} \qquad
       \frac{\partial}{\partial s}\mathbf{x}^*(s) = \mathbf{f}\left(\begin{bmatrix}
           m(s) \\
           \ell(s) \\
         \end{bmatrix}\right) &= \begin{bmatrix}
           p-q \\
           pf-qr-q(1-r)\frac{\ell(s)}{m(s)} \\
         \end{bmatrix}
\end{align*}
We immediately obtain
\[
m(s) = m_0 + (p-q)s.
\]
For $\ell(s)$, substituting the expression for $m(s)$, we have
\begin{align*}
\frac{d}{ds} \ell(s) = pf - qr -q(1-r) \frac{\ell(s)}{m_0 + (p-q)s},
\end{align*}
which implies 
\[
\frac{{\rm d}}{{\rm d}s}\ell(s) + \ell(s) \frac{q(1-r)}{m_0+(p-q)s} = pf - qr.
\]
This is a first-order non-homogeneous linear differential equation of the form
\[
\ell'(s) + \ell(s) a(s) = b(s).
\]
A solution can be written as
\[
\ell(s) = e^{A(s)}(C+v(s)),
\]
where $A(s)= -\int_0^s a(\tau) {\rm d}\tau$ and $v'(s) = e^{-A(s)}b(s)$. Excluding the elementary but lengthy arguments required to compute $A(s)$, $v(s)$, and $C$ we arrive at the solution 
\[
\ell(s) = \ell_0 \left( \frac{m_0}{m_0 + (p-q)s} \right)^\alpha + \left( \frac{f -f^\dagger_c}{1-f^\dagger_c} \right) \left( m_0 + (p-q)s - \frac{m_0^{\alpha+1}}{(m_0 +(p-q)s)^\alpha}\right) ,
\]
where $\alpha = \frac{q(1-r)}{p-q}$.
\end{proof}

The convergence in Lemma \ref{lem:fluid} is illustrated in Figure \ref{fig:scaled_prcess_ode}. Observe that if $m_0, \ell_0 \downarrow 0$ then $m(1) \to (p-q)$ and $\ell(1) \to (p-q)\left( \frac{f-f^\dagger_c}{1-f^\dagger_c} \right)$. We observe that the latter corresponds to the right-hand-side of \eqref{eq:thm32gl}. We now use this insight to establish Theorem \ref{theorem:RPD_Uniformity}.

\begin{figure}[h!]
    \centering
    \makebox[\textwidth]{%
        \includegraphics[width=1\textwidth]{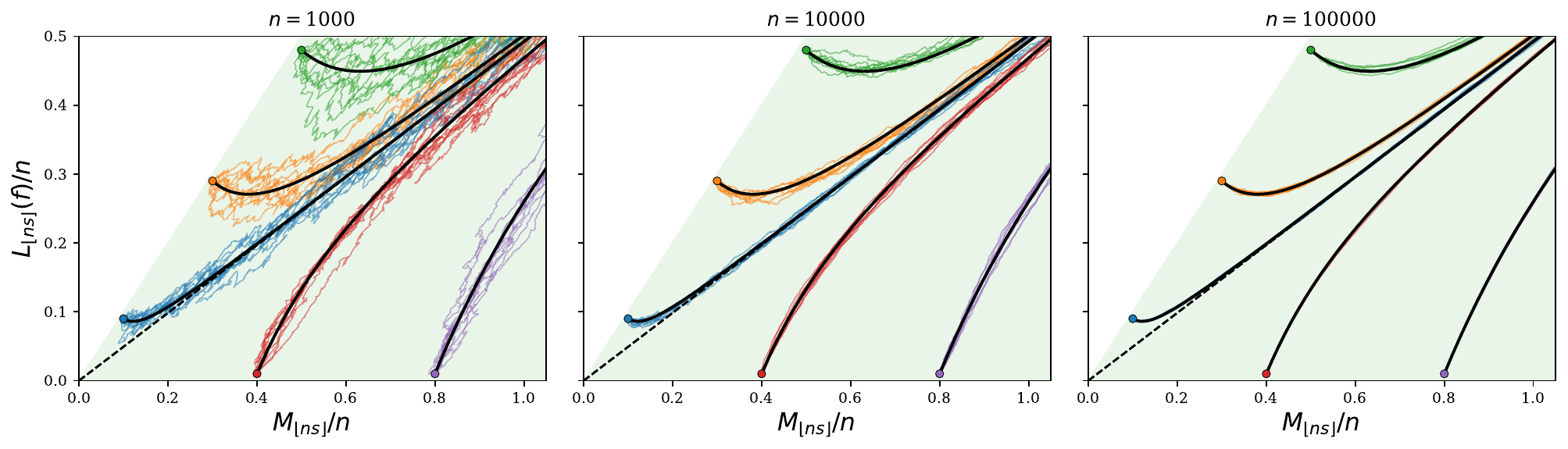}
    }
    \caption{The paths of 10 outcomes of $(\mathbf{X}^n(s))_{s\geq 0}$ for five different initial values $(m_0,\ell_0)$ (colours) with $n=10^3$ (left), $n=10^4$ (centre), and $n=10^5$ (right) when $p=0.55$ and $r=0.5$. The black solid curves represent the fluid limit $(\mathbf{x}(s))_{s\geq 0}$ and the black dashed line displays the line $\ell=\left( \frac{f-f^\dagger_c}{1-f^\dagger_c}\right)m$.}
    \label{fig:scaled_prcess_ode}
\end{figure}

\begin{proof}[Proof of Theorem \ref{theorem:RPD_Uniformity}]
By Lemma \ref{lem:0eps}, for any $\varepsilon, \gamma \in (0,1)$ there exists $\eta(\varepsilon,\gamma)>0$ such that 
\[
\mathbb{P}\left(\max_{0 \leq t \leq \varepsilon} \frac{L_{\lfloor t n \rfloor}(f)}{n} > \eta( \varepsilon,\gamma) \right) \geq 1-\gamma
\]
for $n$ sufficiently large. Let 
\[
\tau(\eta) := \inf \left\{t \in [0, \varepsilon]: \frac{L_{\lfloor t n \rfloor}(f)}{n} > \eta( \varepsilon,\gamma) \right\} \vee \epsilon
\]
By Lemma \ref{lemma:finding_transient_c}, $N_{\lfloor n \tau(\eta) \rfloor}/n \to (p-q)\tau(\eta)$ almost surely.
By the strong Markov property, given $\tau(\eta)$ such that $\tau(\eta)\leq\varepsilon$, we can apply Lemma \ref{lem:fluid} to obtain
\begin{equation}
\begin{aligned}\label{eq:prthm32}
\frac{L_n(f)}{n} &\stackrel{\mathbb{P}}{\to} \eta(\varepsilon, \gamma)\left( \frac{(p-q)\tau(\eta)}{(p-q)\tau(\eta) + (p-q)(1-\tau(\eta))}\right)^{\alpha} \\
&\quad+ \left( \frac{f-f^\dagger_c}{1-f^\dagger_c} \right) \left( (p-q)\tau(\eta) + (p-q)(1-\tau(\eta) ) - \frac{((p-q)\tau(\eta))^{\alpha+1}}{((p-q)\tau(\eta) + (p-q)(1-\tau(\eta) ))^\alpha} \right),
\end{aligned}
\end{equation}
where we set $s=1-\tau(\eta)$, $m_0=(p-q)\tau(\eta)$, and $\ell_0=\eta(\varepsilon,\gamma)$.
Because this holds for any $\gamma,\varepsilon>0$ we then take $\gamma, \varepsilon \downarrow 0$. Noting that since $\tau(\eta) \in [0,\varepsilon]$ this implies $\tau(\eta) \downarrow 0$. Furthermore, the event $\tau(\eta)<\varepsilon$ which we conditioned on in \eqref{eq:prthm32} occurs with probability at least $1-\gamma \uparrow 1$. Taking the limit as  $\gamma, \; \varepsilon\downarrow 0$ in \eqref{eq:prthm32} we then obtain
\[
\frac{L_n(f)}{n} \stackrel{\mathbb{P}}{\to} \left( \frac{f-f^\dagger_c}{1-f^\dagger_c} \right)(p-q)
\]
and hence the result.
\end{proof}

\begin{proof}[Proof of Theorem \ref{theorem:RPD_L_n} (iii)]
Consider the Markov chain $\{\mathbf{X}^{n}(i): i \in \mathbb{N}_0 \}$. If we suppose that $(m_0,\ell_0)^\top= \mathbf{X}^n(i-1)$ and let $\mathbf{x}(1)$ be characterised by \eqref{eq:Xdeflem}, then we can apply Lemma \ref{lem:fluid} (see also the proof of Theorem \ref{theorem:RPD_Uniformity} where we deal with the boundary case $\ell_0=0$) to show that for any $\delta,\varepsilon>0$ there exists $n$ sufficiently large such that 
\[
\mathbb{P}(\lVert \mathbf{X}^n(i) - \mathbf{x}(1) \rVert > \delta)<\varepsilon.
\]
From \eqref{eq:Xdeflem}, we can further observe that ${x}_2(1)$ is increasing in $\ell_0$, and that if $\ell_0 \leq \left( \frac{f-f^\dagger_c}{1-f^\dagger_c} \right) m_0$ then $x_2(1)-x_2(0) \geq \frac{f-f^\dagger_c}{1-f^\dagger_c}$. The latter can be observed visually in Figure \ref{fig:scaled_prcess_ode}, where the dashed line has gradient $\frac{f-f^\dagger_c}{1-f^\dagger_c}$ and the curves that start below this line grow faster than $\frac{f-f^\dagger_c}{1-f^\dagger_c}$. Consequently, applying the Markov property, for $n$ sufficiently large, $\{X^n_2(i): i \in \mathbb{N}_0\}$ stochastically dominates $S_i = \sum_{j=1}^i \xi_i - \inf_{k \leq i}\left\{  \sum_{j=1}^k \xi_i\right\}$ with
\[
\xi_i =
\begin{cases}
\frac{f-f^\dagger_c}{1-f^\dagger_c}-\delta, \quad &\text{with probability } 1- \varepsilon \\
-1 ,&\text{with probability } \varepsilon.
\end{cases}
\]
Here the $-1$ in the second row refers to the fact that at most one species can die each time-step and hence at most $n$ species can die in $n$ time steps.
Noting that $\frac{f-f^\dagger_c}{1-f^\dagger_c}>0$, we can take $\epsilon$ and $\delta$ sufficiently small (and therefore $n$ sufficiently large), to obtain $\mathbb{E}(\xi_i)>0$ making $\{S_i : i \in \mathbb{N}_0\}$ transient.  Because $\{ X^n_2(i): i \in \mathbb{N}_0 \}$ is dominated from below by a transient random walk, it is therefore also transient. The result then follows by noting that since $L_n(f)$ can only decrease by 1 at each time step, which further implies that $\{ L_n(f) : n \in \mathbb{N}_0\}$ cannot only hit 0 finitely often. 
\end{proof}

\begin{proof}[Proof of Theorem \ref{theorem:RPD_survival}]
Suppose a species with fitness $f$ is born at time $b \in \mathbb{N}_0$. Let $S_n$ be the event that this species is alive at time $n > b$, and let $E_n$ be the event that it is killed at time $n$ (given that it has survived until time $n-1$). 
Note that
\[
\mathbb{P}(E_n \mid S_{n-1},\dots,S_{b}) =\mathbb{P}(E_n \mid S_{n-1})
\geq q(1-r)\frac{1}{M_n},
\]
because in the “random-death’’ event (probability $q(1-r)$) the species is chosen uniformly from the $M_n$ individuals alive at that time. Since $M_0=0$ and the population increases by at most one per time-step, we have $M_n\le n$.
Thus, for $r<1$,
\begin{equation}\label{eqn:rpdsurb}
\mathbb{P}(E_n \mid S_{n-1},\dots,S_{b+1}) =\mathbb{P}(E_n \mid S_{n-1})
\ge \frac{q(1-r)}{n}>0.
\end{equation}
The survival probability up to time $n^*>b$ is
\[
\mathbb{P}(S_{b+1}\cap\dots\cap S_{n^*})
= \prod_{n=b+1}^{n^*} \mathbb{P}(S_n \mid S_{b+1},\dots,S_{n-1}).
\]
Using $\mathbb{P}(S_n \mid \cdot) = 1 - \mathbb{P}(E_n \mid \cdot)$ and the bound in \eqref{eqn:rpdsurb} we obtain
\[
\mathbb{P}(S_{b+1}\cap\dots\cap S_{n^*})
\;\le\;
\prod_{n=b+1}^{n^*} \left(1 - \frac{q(1-r)}{n}\right).
\]
It is well known that
$
\prod_{n=b+1}^\infty \left(1 - \frac{c}{n}\right) = 0
$
for any $c>0$ and $b \in \mathbb{N}_0$.
Therefore,
\[
\mathbb{P}(S_{b+1}\cap S_{b+2}\cap\dots ) = 0,
\]
meaning that the species with fitness $f$ dies almost surely.  
Since this holds for any species fitness $f \in [0,1]$ and birth-time $b \in \mathbb{N}_0$, when $r<1$ the probability that any species survives forever is zero.
\end{proof}

\subsection{Proofs for the RLD process}

\begin{proof}[Proof of Theorem \ref{theorem:RLD_rec}]
Recall the construction of the RLD process in Section \ref{sec:RLDdef}. Using the same random variables as defined in that section (i.e., on the same probability space) we construct the following related process which is equivalent to the RPD process except that, if a species dies and it is a random kill (i.e., with probability $q(1-r)$), then if $U_n \in [0,f]$ kill the least fit species, and otherwise if $U_n \in (f,1]$ kill no species. Letting $L^*_n(f)$ denote the number of species with fitness in $[0,f]$ in this process (and $L_n(f)$ the number in the RLD process), following a similar line of reasoning as in the proof of Theorem \ref{theorem:RPD_L_n} (i), we can verify that $L_n(f) \leq L^*_n(f)$ for all $n$ surely. Moreover, $\{ L^*_n(f): n \in \mathbb{N}_0\}$ is a 1-D birth death Markov process whose transition probabilities correspond to those of the GMS with modified probabilities; specifically, with reference to Figure \ref{fig:|L_n(f)|_gms}, the process has the same step-up probability $pf$, it has a different step down probability $q(r + f(1-r))$ (in place of $q$ in Figure \ref{fig:|L_n(f)|_gms}), and it has a different remain in place probability $1-pf-q(r + f(1-r))$ (in place of $p(1-f)$ in Figure \ref{fig:|L_n(f)|_gms}). The Markov chain $\{ L^*_n(f): n \in \mathbb{N}_0\}$ is hence recurrent if and only if 
\[
pf \geq q(r + f(1-r))
\]
which is equivalent to 
\[
f \leq  \frac{qr}{p-q(1-r)} = f^\ddagger_c.
\]
The fact that when $f<f^\ddagger_c$ we have $E(B_{i+1}-B_i) \leq f^\ddagger_c/(f^\ddagger_c-f)$ then follows from the same arguments applied in the proof of Theorem \ref{theorem:RPD_L_n}(i).
\end{proof}

\begin{proof}[Proof of Theorem \ref{theorem:RLD_survival}]
Suppose a species with fitness $f \in [0,1)$ exists in the population and $r < 1$. 
We first note that the RLD process can be equivalently generated by the following dynamics. At each time step $n \geq 1$: independently of everything else generate $U_n \sim \text{Uniform}(0,1)$ and $E_n$ where $E_n$ is a random variable on $\{B,D,L\}$ with $\mathbb{P}(E_n=B)=p$, $\mathbb{P}(E_n=L)=rq$, and $\mathbb{P}(E_n=R)=(1-r)q$;
\begin{itemize}
\item if $E_n=B$ then a species is born with fitness $U_n$
\item if $E_n=L$ the least fit species is killed 
\item if $E_n=R$ then the species with the highest fitness less than $U_n$ is killed (if no such species exists kill the least fit).
\end{itemize}
Let $\mathcal{N} = \{ n \in \mathbb{N}: U_n \geq f\}$ and note that by the Borel-Cantelli lemma $|\mathcal{N}| =\infty$ almost surely. Consider the set of record minima times of $U_n$ given $U_n\geq f$, that is, 
\[
\mathcal{RM}:=\{ n \in \mathcal{N} : U_n < U_k \text{ for all } k \in \mathcal{N}  \text{ with } k < h\}.
\]
Using the fact that the $(U_n)_{n\in\mathbb{N}}$ are independent and applying Rényi's theorem on record values \cite{Renyi1962} we conclude that $|\mathcal{RM}|=\infty$ almost surely. Since $E_n$ and $U_n$ are independent, for each $n \in \mathcal{R}$, we have $\mathbb{P}(E_n=R)=(1-r)q>0$. Using the independence of the $(E_n)_{n \in \mathbb{N}}$, with each other and everything else, we conclude there almost surely exists $n \in \mathcal{RM}$ with $E_n=R$. Note that if the species with fitness $f$ is still alive at time step $n-1$, then if $n \in \mathcal{RM}$ and $E_n =R$ then the species with fitness $f$ must be killed at time step $n$. Hence any arbitrary species with fitness $f \in [0,1)$ is killed almost surely.
\end{proof}

\begin{proof}[Proof of Proposition \ref{prop:grp}]
Wlog we assume that at time 0 there are $k$ species with fitness strictly less than $1$ and one species with fitness 1. Observe that the species with fitness 1 is eventually killed if and only if there exists $n^* \in \mathbb{N}$ such that $M_{n^*}=0$. Since $\{M_n: n \in \mathbb{N}_0\}$ is a random walk with a reflecting barrier at zero that increases by 1 with probability $p$ and decreases by 1 with probability $q$, this corresponds to the well known gambler's ruin probability, leading to the result.
\end{proof}


\begin{thebibliography}{99}

\bibitem{BakSneppen1993}
P. Bak, K. Sneppen. Punctuated equilibrium and criticality in a simple model of evolution.
\emph{Physics Review Letters} 71 (1993) 4083–4086.

\bibitem{B13}
I. Ben-Ari. An empirical process interpretation of a model of species survival. \emph{Stochastic Processes and their Applications} 123 (2013) 475--489.

\bibitem{BMR11} I. Ben-Ari, A. Matzavinos, A. Roitershtein. On a species survival model. \emph{Electronic Communications in Probability} 16 (2011) 226-233.

\bibitem
{Quasispecies Model} I. Ben-Ari, R.B. Schinazi.
A stochastic model for the evolution of a quasispecies.
\emph{Journal of Statistical Physics} 162 (2016) 415–425.

\bibitem{Actual Process} D. Bertacchi, J. Lember, F. Zucca. A stochastic model for the evolution of species with random fitness. \emph{Electronic Communications in Probability} 23 (2018) 1–13. 

\bibitem{Influenza Process} J. T. Cox, R. B. Schinazi.
A stochastic model for the evolution of the influenza virus. \emph{Markov Processes and Related Fields} 20 (2014) 155-166. 

\bibitem{EventBasedExtinction}
L.R. Fontes, F.S. Marques, C. Grejo.
An evolution model with event-based extinction.
\emph{Journal of Physics A: Mathematical and Theoretical} 53 (2020) 19.

\bibitem{StrongestFitness}
C. Grejo, F.P. Machado, A. Roldan-Correa.
 The fitness of the strongest individual in the subcritical GMS model. 
\emph{Electronic Communications in Probability} 21 (2016).

\bibitem{Simple Process} H. Guiol, F.P. Machado, R.B. Schinazi.
A stochastic model of evolution. \emph{Markov Processes and Related Fields} 17 (2011) 253-258.

\bibitem{Bessel Link} H. Guiol, F.P. Machado, R.B. Schinazi. On a link between a species survival time in an evolution model and the Bessel distributions.
\emph{Brazilian Journal of Probability and Statistics} 27 (2013) 201–209.

\bibitem
{FluidLimitKurtz}
T.G. Kurtz. Solutions of ordinary differential equations as limits of pure jump Markov processes.
\emph{Journal of Applied Probability} 7 (1970) 49--58.

\bibitem{L85}
T.M. Liggett. 
Interacting particle systems. \emph{Springer} (1985).

\bibitem
{Phylogenetic Trees} T.M. Liggett, R.B. Schinazi.
A stochastic model for phylogenetic trees.
\emph{Journal of Applied Probability} 46 (2009) 601–607.

\bibitem{MZ03}
R. Meester, D. Znamenski. Limit behavior of the Bak--Sneppen evolution model. \emph{The Annals of Probability} 31 (2003) 1986-2002.

\bibitem{MS12}
S. Michael, S.E. Volkov. 
\newblock On the generalization of the GMS evolutionary model. 
\emph{ Markov Processes and Related Fields} 18 (2012) 311--322.

\bibitem
{Paleontological Extinction}
D.M. Raup. Extinction from a paleontological perspective.
\emph{European Review} 1 (1993) 207--216.

\bibitem{Renyi1962}
A. R\'enyi.
Th\'eorie des \'el\'ements saillants d'une suite d'observations. 
\emph{Annales scientifiques de l'Université de Clermont. Mathématiques} 
8 (1962) 7--13. 

\bibitem
{Subspecies Model} R. Roy, H. Tanemura.
On a model of evolution of subspecies.
\emph{Journal of Mathematical Biology} 90 (2025) 3


\end{thebibliography}
\end{document}